\documentclass{amsart}

\title{Symmetric and nonsymmetric Hall--Littlewood polynomials of type BC}
\author{Vidya Venkateswaran}
\address{Department of Mathematics, MIT, Cambridge, MA 02139}
\email{\href{mailto:vidyav@math.mit.edu}{vidyav@math.mit.edu}}
\thanks{Research supported by NSF Mathematical Sciences Postdoctoral Research Fellowship DMS-1204900}
\subjclass[2000]{33D52,33D45}
\keywords{Koornwinder polynomials, orthogonal polynomials, symmetric functions, Hecke algebras}
\usepackage{hyperref}
\usepackage[svgnames]{xcolor}
\hypersetup{colorlinks,breaklinks,
            linkcolor=DarkBlue,urlcolor=DarkBlue,
            anchorcolor=DarkBlue,citecolor=DarkBlue}
\usepackage{url}
\usepackage{euscript}
\usepackage[verbose,letterpaper,tmargin=1in,bmargin=1in,lmargin=1in,rmargin=1in]{geometry}
\usepackage{amssymb}
\usepackage{amsmath}

\usepackage{xypic}
\usepackage{verbatim}

\newtheorem{theorem}{Theorem}[section]
\newtheorem{claim}{Claim}[theorem]

\newtheorem{lemma}[theorem]{Lemma}
\newtheorem{proposition}[theorem]{Proposition}
\newtheorem{corollary}[theorem]{Corollary}

\theoremstyle{definition}
\newtheorem{definition}[theorem]{Definition}

\theoremstyle{remark}
\newtheorem*{remarks}{Remarks}

\begin{document}

\begin{abstract}
Koornwinder polynomials are a $6$-parameter $BC_{n}$-symmetric family of Laurent polynomials indexed by partitions, from which Macdonald polynomials can be recovered in suitable limits of the parameters.  As in the Macdonald polynomial case, standard constructions via difference operators do not allow one to directly control these polynomials at $q=0$.  In the first part of this paper, we provide an explicit construction for these polynomials in this limit, using the defining properties of Koornwinder polynomials.  Our formula is a first step in developing the analogy between Hall-Littlewood polynomials and Koornwinder polynomials at $q=0$.  In the second part of the paper, we provide an analogous construction for the nonsymmetric Koornwinder polynomials in the same limiting case.  The method employed in this paper is a $BC$-type adaptation of techniques used in an earlier work of the author, which gave a combinatorial method for proving vanishing results of Rains and Vazirani at the Hall-Littlewood level.  As a consequence of this work, we obtain direct arguments for the constant term evaluations and norms in both the symmetric and nonsymmetric cases.  
\end{abstract}

\maketitle

\section{Introduction}

In \cite{MacO}, Macdonald introduced a very important family of multivariate $q$-orthogonal polynomials associated to a root system.  These polynomials, and their connections to representation theory, combinatorics and algebra, have been well-studied and are an active area of research.  For the type $A$ root system, Macdonald polynomials $P_{\lambda}(x_{1}, \dots, x_{n};q,t)$ contain many well-known families of symmetric functions as special cases: for example, the Schur, Hall-Littlewood, and Jack polynomials occur at $q=t$, $q=0$, and $t = q^{\alpha}, q \rightarrow 1$, respectively.  The existence of the top level Macdonald polynomials was proved by exhibiting a suitable operator which has these polynomials as its eigenfunctions.  A particularly important degeneration of the Macdonald polynomials is obtained in the $q=0$ limit where one obtains zonal spherical functions on semisimple $p$-adic groups.  In fact, Macdonald provides an explicit formula for the spherical functions of the Chevalley group $G(\mathbb{Q}_{p})$  in terms of the root data for the group $G$ \cite{MacP}.  In particular, this generalizes the formula for the Hall-Littlewood polynomials \cite[Ch. III]{Mac}, which arise as zonal spherical functions for $Gl_{n}(\mathbb{Q}_{p})$.  

In \cite{K}, Koornwinder introduced a remarkable class of multivariate $q$-orthogonal polynomials associated to the non-reduced root system $BC_{n}$.  These polynomials are a $6$-parameter family of Laurent polynomials which are invariant under permuting variables and taking inverses of variables.  Moreover, these polynomials reduce to the Askey-Wilson polynomials at $n=1$, and one recovers the Macdonald polynomials by taking suitable limits of the parameters \cite{vD}.  As in the Macdonald polynomial case, the existence of these polynomials was proved by using $q$-difference operators; these behave badly as $q \rightarrow 0$.  The explicit construction for the Macdonald polynomials at $q=0$ is due to Littlewood; in fact, Hall also provided a construction of these polynomials indirectly via the Hall algebra.  Given the relationship between Macdonald and Koornwinder polynomials, a natural question one can ask is whether there exists an explicit construction for the latter polynomials at $q=0$, thereby providing an analog of the construction of Hall-Littlewood polynomials for this family.  In this work, we use the defining properties of Koornwinder polynomials to provide a closed formula at the $q=0$ limit.  We then extend this technique to study the nonsymmetric Koornwinder polynomials in the same limit.  We provide an explicit formula when these polynomials are indexed by \textit{partitions}; we then use elements of the affine Hecke algebra of type $BC$ to recursively obtain \textit{all} nonsymmetric Koornwinder polynomials in this limit.  A nice feature of this work is the self-contained proofs of the constant term evaluations and norm evaluations in both the symmetric and nonsymmetric cases (Theorems \ref{symmeval}, \ref{nonsymmeval} and Theorems \ref{symmorthognorm}, \ref{nonsymmnorm}).  We mention that, in the symmetric case, the constant term evaluation at the $q$-level is a famous result of Gustafson \cite{G}; this in turn is a multivariate generalization of a result of Askey and Wilson \cite{AW} and a $q$-generalization of Selberg's beta integral \cite{S}.  We note that Gustafson's approach requires $q \neq 0$, so one cannot directly apply that argument in this limiting case.  

The motivation for this problem arose when the author was investigating direct proofs at the Hall-Littlewood level for the vanishing results of Rains and Vazirani \cite{RV} (note that many of those results were first conjectured in \cite{R}).  These identities are $(q,t)$-generalizations of restriction rules for Schur functions.  More precisely, one integrates a suitably specialized Macdonald polynomial indexed by $\lambda$ against a particular density; the result vanishes unless $\lambda$ satisfies an explicit condition and at $q=t$ one recovers an identity about Schur functions.  In \cite{VV}, we provided a combinatorial technique for proving these results at the Hall-Littlewood level; this also led to various generalizations.  This method makes use of the structure of the Hall-Littlewood polynomials as a sum over the Weyl group.  Many of the results in \cite{RV} involve Koornwinder polynomials so, in order to extend the method developed in that paper, it is necessary to first find a closed formula at the $q=0$ level.  In fact, the technique used in this paper to prove that the specified polynomials are orthogonal with respect to the Koornwinder density is a natural adaptation of the ideas in \cite{VV} to the type $BC_{n}$ case.  

This paper has two main components: the first part deals with the symmetric theory at $q=0$, while the second deals with the nonsymmetric theory in the same limit.  The first section of each part sets up the relevant notation, reviews some background material, and defines the polynomials in question.  The second section of each part consists of the main theorems and proofs, in particular we prove that these are indeed the symmetric and nonsymmetric Koornwinder polynomials at $q=0$, respectively.  As mentioned above, \cite{RV} provided vanishing results involving suitably specialized (symmetric) Koornwinder polynomials.  Thus, the third section of the first part contains an application of our formula to this work: we use the construction of the Koornwinder polynomials at $q=0$ to  provide a new combinatorial proof of a result from \cite{RV}.  

\medskip
\noindent\textbf{Acknowledgements.} The author would like to thank E. Rains for suggesting the topics in this paper and for numerous helpful conversations throughout this work.

\section{Symmetric Hall-Littlewood polynomials of type BC} 
\subsection{Background and Notation} \label{notation}

We will first review some relevant notation before introducing the polynomials that are these subject of this paper; a good reference is \cite[Ch. 1]{Mac}.

Recall that a partition $\lambda$ is a non-increasing string of positive integers $(\lambda_{1}, \lambda_{2}, \dots, \lambda_{n})$, in which some of the $\lambda_{i}$ may be zero.  We call the $\lambda_{i}$ the ``parts" of $\lambda$.  We write $l(\lambda) = \text{max}\{k \geq 0| \lambda_{k} \neq 0 \}$ (the ``length") and $|\lambda| = \sum_{i=1}^{n} \lambda_{i}$ (the ``weight").  A string $\mu = (\mu_{1}, \dots, \mu_{n})$ of integers (not necessarily non-increasing or positive) is called a composition of $|\mu| = \sum_{i=1}^{n} \mu_{i}$.  We will say $\lambda$ is an ``even partition" if all parts of $\lambda$ are even; in this case we use the notation $\lambda = 2\mu$ where $\mu_{i} = \lambda_{i}/2$ for all $i$.  We will also say $\lambda$ has ``all parts occurring with even multiplicity" if the conjugate partition $\lambda'$ is an even partition.  A composition $\lambda$ is an element of $\mathbb{Z}^{n}$ for some $n \geq 1$; we will denote this set by $\Lambda$.

We briefly recall some orderings on compositions.
\begin{definition}
Let $\leq$ denote the dominance partial ordering on compositions, i.e., $\mu \leq \lambda$ if and only if $$\sum_{1 \leq i \leq k} \mu_{i} \leq \sum_{1 \leq i \leq k} \lambda_{i}$$ for all $k \geq 1$ (and $\mu< \lambda$ if $\mu \leq \lambda$ and $\mu \neq \lambda$).  Let $\stackrel{\text{lex}}{\leq}$ denote the reverse lexicographic ordering: $\mu \stackrel{\text{lex}}{\leq} \lambda$ if and only if $\lambda = \mu$ or the first non-vanishing difference $\lambda_{i} - \mu_{i}$ is positive.  \end{definition}
Note that $\stackrel{\text{lex}}{\leq}$ is a total ordering.  

\begin{lemma} \label{order}
Let $\mu, \lambda \in \mathbb{Z}^{n}$ such that $\mu \leq \lambda$.  Then $\mu \stackrel{\text{lex}}{\leq} \lambda$.  
\end{lemma}
\begin{proof}
The claim is clearly true if $\mu = \lambda$, so suppose $\mu < \lambda$.  If $\mu_{1} < \lambda_{1}$, we're done; otherwise $\mu_{1} = \lambda_{1}$ and $\mu_{2} \leq \lambda_{2}$ since $\mu_{1} + \mu_{2} \leq \lambda_{1} + \lambda_{2}$.  Iterating this argument produces an integer $i$ in $\{1, \dots, n\}$ such that $\mu_{1} = \lambda_{1}, \dots, \mu_{i-1} = \lambda_{i-1}$ and $\mu_{i} < \lambda_{i}$.  Thus, $\mu \stackrel{\text{lex}}{<} \lambda$ as desired.
\end{proof}

\begin{definition} \label{nonsymmorder}
Let $\mu$ and $\lambda$ be two elements of $\mathbb{Z}^{n}$.  We will write $\mu^{+}$ for the unique dominant weight in the $BC_{n}$ orbit of $\mu$ (that is, the partition obtained by rearranging the absolute values of the parts of $\mu$ in non-increasing order).  Then we write $\mu \prec \lambda$ if and only if either 1) $\mu^{+} < \lambda^{+}$ or if 2) $\mu^{+} = \lambda^{+}$ and $\mu \leq \lambda$, and in either case $\mu \neq \lambda$. 
\end{definition}

\begin{remarks}
This part will mostly deal with partitions and the dominance and reverse lexicographic orderings.  Compositions, and the extended dominance ordering appearing in Definition \ref{nonsymmorder}, will become relevant in the following part that deals with the nonsymmetric theory.  
\end{remarks}

Let $m_{i}(\lambda)$ be the number of $\lambda_{j}$ equal to $i$ for each $i \geq 0$.  Then we define:  
\begin{equation}\label{v}
v_{\lambda}(t;a,b;t_{0}, \dots, t_{3}) = \Bigg( \prod_{i \geq 0} \prod_{j=1}^{m_{i}(\lambda)} \frac{1- t^{j}}{1-t} \Bigg) \prod_{i=1}^{m_{1}(\lambda)} (1- t_{0}t_{1}t_{2}t_{3} t^{i-1+2m_{0}(\lambda)}) \prod_{i=1}^{m_{0}(\lambda)} (1-ab t^{i-1}),
\end{equation}
and
\begin{equation} \label{vplus}
v_{\lambda+}(t;t_{0}, \dots, t_{3}) = \Bigg( \prod_{i \geq 1} \prod_{j=1}^{m_{i}(\lambda)} \frac{1- t^{j}}{1-t} \Bigg) \prod_{i=1}^{m_{1}(\lambda)} (1- t_{0}t_{1}t_{2}t_{3} t^{i-1+2m_{0}(\lambda)}).
\end{equation}
Note the comparison with the factors making the Hall-Littlewood polynomials monic in \cite[Ch.  III]{Mac}.  Also note that 
\begin{equation*}
v_{\lambda}(t;a,b;t_{0}, \dots, t_{3}) = v_{\lambda+}(t;t_{0}, \dots, t_{3}) v_{0^{m_{0}(\lambda)}}(t;a,b;t_{0}, \dots, t_{3}).
\end{equation*}

Throughout this paper, we will use 
\begin{align*}
T =T_{n} &= \{(z_{1}, \dots, z_{n}) : |z_{1}| = \dots = |z_{n}| = 1 \}, \\
dT &= \prod_{1 \leq j \leq n} \frac{dz_{j}}{2\pi \sqrt{-1}z_{j}}
\end{align*}
to denote the $n$-torus and Haar measure, respectively.  Since many of the objects we will be dealing with are functions of $n$ variables, we will often use the superscript ${(n)}$ with $z$ in the argument, instead of $(z_{1}, \dots, z_{n})$.  We define the $q$-symbol
\begin{equation*}
(a;q) = \prod_{k \geq 0} (1-aq^{k})
\end{equation*}
and let $(a_{1},a_{2}, \dots, a_{l};q)$ denote $(a_{1};q)(a_{2};q) \cdots (a_{l};q)$.

We recall the symmetric Koornwinder density \cite{K}:
\begin{equation*}
\tilde \Delta_{K}^{(n)}(z;q,t;t_{0}, \dots, t_{3}) 
= \frac{(q;q)^{n}}{2^{n}n!} \prod_{1 \leq i \leq n} \frac{(z_{i}^{\pm 2};q)}{(t_{0}z_{i}^{\pm 1}, t_{1}z_{i}^{\pm 1}, t_{2}z_{i}^{\pm 1}, t_{3}z_{i}^{\pm 1};q)} \prod_{1 \leq i<j \leq n} \frac{(z_{i}^{\pm 1}z_{j}^{\pm 1};q)}{(tz_{i}^{\pm 1}z_{j}^{\pm 1};q)}.
\end{equation*}
Since we are concerned with $q=0$ degenerations of Koornwinder polynomials, we will be interested in the symmetric Koornwinder density in the same limiting case:
\begin{multline} \label{koorndensity}
\tilde \Delta_{K}^{(n)}(z;0,t;t_{0},t_{1},t_{2},t_{3}) \\
= \frac{1}{2^{n}n!} \prod_{1 \leq i \leq n} \frac{(1-z_{i}^{\pm 2})}{(1-t_{0}z_{i}^{\pm 1})(1-t_{1}z_{i}^{\pm 1})(1-t_{2}z_{i}^{\pm 1})(1-t_{3}z_{i}^{\pm 1})} \prod_{1 \leq i<j \leq n} \frac{(1-z_{i}^{\pm 1}z_{j}^{\pm 1})}{(1-tz_{i}^{\pm 1}z_{j}^{\pm 1})},
\end{multline}
where we write $(1-z_{i}^{\pm 2})$ for the product $(1-z_{i}^{2})(1-z_{i}^{-2})$ and $(1-z_{i}^{\pm 1}z_{j}^{\pm 1})$ for $(1-z_{i}z_{j})(1-z_{i}^{-1}z_{j}^{-1})(1-z_{i}^{-1}z_{j})(1-z_{i}z_{j}^{-1})$, etc.  We will write $\tilde \Delta_{K}^{(n)}(z;t;t_{0}, \dots, t_{3})$ to denote this density.  

Using this density, we let
\begin{equation} \label{norm}
N_{\lambda}(t;t_{0}, \dots, t_{3}) = \frac{1}{v_{\lambda+}(t)} \int_{T} \tilde \Delta_{K}^{(m_{0}(\lambda))}(z;t;t_{0}, \dots, t_{3}) dT. 
\end{equation}
We note that, at the $q$-level, the explicit evaluation of the integral above is a famous result of Gustafson \cite{G}.  However, the arguments do not directly apply at $q=0$.  In keeping with the theme of this work, we will provide a self-contained proof of the evaluation of this integral in Theorem \ref{symmeval}.  This will provide an explicit formula for the quantity $N_{\lambda}(t;t_{0}, \dots, t_{3})$.  

For simplicity of notation, we will write $v_{\lambda}, v_{\lambda+}, N_{\lambda}, \tilde \Delta_{K}^{(n)}$, etc., when the parameters are clear from the context.

Finally, we explain some notation involving elements of the hyperoctahedral group, $B_{n}$.  An element in $B_{n}$ is determined by specifying a permutation $\rho \in S_{n}$ as well as a sign choice $\epsilon_{\rho}(i)$, for each $1 \leq i \leq n$.  Thus, $\rho$ acts on the subscripts of the variables, for example by
\begin{equation*}
\rho(z_{1} \cdots z_{n}) = z_{\rho(1)}^{\epsilon_{\rho}(1)} \cdots z_{\rho(n)}^{\epsilon_{\rho}(n)}.
\end{equation*}
If $\rho(i) = 1$, we will say that $z_{1}$ occurs in position $i$ of $\rho$.  We also write
\begin{equation*}
``z_{i} \prec_{\rho} z_{j}"
\end{equation*}
if $i = \rho(i')$ and $j = \rho(j')$ for some $i'<j'$, i.e., $z_{i}$ appears to the left of $z_{j}$ in the permutation $z_{\rho(1)}^{\epsilon_{\rho}(1)} \cdots z_{\rho(n)}^{\epsilon_{\rho}(n)}$.  We also define $\epsilon_{\rho}(z_{i})$ to be $\epsilon_{\rho}(i')$ if $i = \rho(i')$, i.e., it is the exponent  $(\pm 1)$ on $z_{i}$ in  $z_{\rho(1)}^{\epsilon_{\rho}(1)} \cdots z_{\rho(n)}^{\epsilon_{\rho}(n)}$.

We finally define the main objects of this section.  
\begin{definition}
Let $\lambda$ be a partition with $l(\lambda) \leq n$ and $|a|, |b|, |t|, |t_{0}|, \dots, |t_{3}| < 1$.  Then $K_{\lambda}(z_{1}, \dots, z_{n};t;a,b;$ $t_{0}, \dots, t_{3})$, indexed by $\lambda$, is defined by
\begin{align} \label{K-pol}
 \frac{1}{v_{\lambda}(t;a,b;t_{0}, \dots, t_{3})} \sum_{w \in B_{n}} w \Bigg( \prod_{1 \leq i \leq n} u_{\lambda}(z_{i}) \prod_{1 \leq i<j \leq n} \frac{1-tz_{i}^{-1}z_{j}}{1-z_{i}^{-1}z_{j}} \frac{1-tz_{i}^{-1}z_{j}^{-1}}{1-z_{i}^{-1}z_{j}^{-1}} \Bigg),
\end{align}
where 
\begin{align*}
u_{\lambda}(z_{i}) &= \begin{cases} \frac{(1-az_{i}^{-1})(1-bz_{i}^{-1})}{1-z_{i}^{-2}} & \text{if $\lambda_{i} = 0$,}
\\
z_{i}^{\lambda_{i}} \frac{(1-t_{0}z_{i}^{-1})(1-t_{1}z_{i}^{-1})(1-t_{2}z_{i}^{-1})(1-t_{3}z_{i}^{-1})}{1-z_{i}^{-2}} &\text{if $\lambda_{i} >  0$.}
\end{cases}
\end{align*}
\end{definition}

\begin{remarks}
We note that the $K_{\lambda}$ are actually independent of $a,b$ - this is a scaling factor accounted for in $v_{\lambda}$.  In particular, the arguments below that prove that this is indeed the Koornwinder polynomial at $q=0$ work for any choice of $a,b$.  However, we leave in arbitrary $a,b$ (as opposed to the choice $\pm 1$) because the resulting form is useful in applications; an example that illustrates this appears in the last section.

\end{remarks}
We will also let
\begin{equation} \label{R-pol}
R_{\lambda}^{(n)}(z;t;a,b;t_{0}, \dots, t_{3}) = v_{\lambda}(t;a,b;t_{0}, \dots, t_{3}) K_{\lambda}^{(n)}(z;t;a,b;t_{0}, \dots, t_{3}),
\end{equation}
 and for $w \in B_{n}$, we let
\begin{equation} \label{Kterm}
R_{\lambda, w}^{(n)}(z;t;a,b;t_{0}, \dots, t_{3}) = w\Bigg( \prod_{1 \leq i \leq n} u_{\lambda}(z_{i}) \prod_{1 \leq i<j \leq n} \frac{1-tz_{i}^{-1}z_{j}}{1-z_{i}^{-1}z_{j}}\frac{1-tz_{i}^{-1}z_{j}^{-1}}{1-z_{i}^{-1}z_{j}^{-1}} \Bigg)
\end{equation}
be the associated term in the summand.
As usual, we will write $K_{\lambda}^{(n)}, R_{\lambda}^{(n)}$ and $R_{\lambda, w}^{(n)}$ when the parameters are clear from context.  
\begin{remarks}
When $(t_{0}, t_{1}, t_{2}, t_{3}) = (a,b,0,0)$, we obtain
\begin{multline*}
K_{\lambda}(z_{1}, \dots, z_{n};t;a,b;a,b,0,0) \\= \frac{1}{v_{\lambda}(t)} \sum_{w \in B_{n}} w \Bigg( \prod_{1 \leq i \leq n} z_{i}^{\lambda_{i}} \frac{(1-az_{i}^{-1})(1-bz_{i}^{-1})}{1-z_{i}^{-2}} \prod_{1 \leq i<j \leq n} \frac{1-tz_{i}^{-1}z_{j}}{1-z_{i}^{-1}z_{j}} \frac{1-tz_{i}^{-1}z_{j}^{-1}}{1-z_{i}^{-1}z_{j}^{-1}} \Bigg).
\end{multline*}
In particular, this is Macdonald's $2$-parameter family $(BC_{n}, B_{n}) = (BC_{n}, C_{n})$ polynomials at $q=0$.  We will write $K_{\lambda}^{(n)}(z;t;a,b,0,0)$ in this case.  
\end{remarks}

\subsection{Main Results}

In this section, we will show that the $K_{\lambda}^{(n)}(z; t;a,b; t_{0}, \dots, t_{3})$ are indeed the Koornwinder polynomials at $q=0$.

\begin{theorem} \label{BCpol}
The function $K_{\lambda}^{(n)}(z;t;a,b;t_{0}, \dots, t_{3})$ is a $BC_{n}$-symmetric Laurent polynomial (i.e., invariant under permuting variables $z_{1}, \dots, z_{n}$ and inverting variables $z_{i} \rightarrow z_{i}^{-1}$).  
\end{theorem}

\begin{proof}
Recall the fully $BC_{n}$-antisymmetric Laurent polynomials:
\begin{equation} \label{BCanti}
\Delta_{BC} = \prod_{1 \leq i \leq n} z_{i} - z_{i}^{-1} \prod_{1 \leq i<j \leq n} z_{i}^{-1} - z_{j} - z_{j}^{-1} + z_{i}
= \prod_{1 \leq i \leq n} \frac{z_{i}^{2}-1}{z_{i}} \prod_{1 \leq i<j \leq n} \frac{1-z_{i}z_{j}}{z_{i}z_{j}}(z_{j}-z_{i}).
\end{equation}
Then we have
\begin{equation} \label{Kdelta}
K_{\lambda}^{(n)}(z;a,b;t_{0}, \dots, t_{3};t) \cdot \Delta_{BC} = \frac{1}{v_{\lambda}(t)} \sum_{w \in B_{n}} \epsilon(w) w \Big( \prod_{1 \leq i \leq n} u_{\lambda}'(z_{i}) \prod_{1 \leq i<j \leq n} (1-tz_{i}^{-1}z_{j}^{-1})(z_{i}-tz_{j}) \Big),
\end{equation}
where
\begin{align*}
u_{\lambda}'(z_{i}) &= \begin{cases} z_{i}(1-az_{i}^{-1})(1-bz_{i}^{-1}) & \text{if $\lambda_{i} = 0$,}
\\
z_{i}^{\lambda_{i}+1} (1-t_{0}z_{i}^{-1}) \cdots (1-t_{3}z_{i}^{-1}) &\text{if $\lambda_{i} >  0$.}
\end{cases}
\end{align*}
Notice that $K_{\lambda}^{(n)} \cdot \Delta_{BC}$ is a $BC_{n}$-antisymmetric Laurent polynomial, so in particular $\Delta_{BC}$ divides $K_{\lambda}^{(n)} \cdot \Delta_{BC}$ as polynomials.  Consequently, $K_{\lambda}^{(n)}$ is a $BC_{n}$-symmetric Laurent polynomial, as desired.  
\end{proof}

\begin{theorem} \label{symmtriang}
The functions $K_{\lambda}^{(n)}(z;t;a,b;t_{0}, \dots, t_{3})$ are triangular with respect to dominance ordering:
\begin{align*}
K_{\lambda}^{(n)}(z;t;a,b;t_{0}, \dots, t_{3}) &= m_{\lambda} + \sum_{\mu < \lambda} c^{\lambda}_{\mu} m_{\mu}.
\end{align*}
\end{theorem}

\begin{remarks}
Here $\{m_{\lambda}\}_{\lambda}$ is the monomial basis with respect to Weyl group of type $BC$:
\begin{align*}
m_{\lambda} &= \sum_{w \in B_{n}} w( z_{1}^{\lambda_{1}} \cdots z_{n}^{\lambda_{n}}).
\end{align*} 
\end{remarks}

\begin{proof}
We show that when $K_{\lambda}^{(n)}$ is expressed in the monomial basis, the top degree term in $m_{\lambda}$; moreover, it is monic.  First note that from (\ref{BCanti}) in the previous proof, we have
\begin{align*}
\Delta_{BC} &= z^{\rho} + (\text{dominated terms}),
\end{align*}
where $\rho = (n, n-1,  \dots ,2, 1)$.  We compute the dominating monomial in $K_{\lambda}^{(n)} \cdot \Delta_{BC}$; see (\ref{Kdelta}) in the previous proof for the formula.  Note that if $\lambda_{i} = 0$, we have highest degree $\lambda_{i} + 1$ in $u_{\lambda}'(z_{i})$.  Similarly, if $\lambda_{i} > 0$, we note that $\lambda_{i} + 1 \geq -\lambda_{i} + 3$ (with equality if and only if $\lambda_{i} = 1$) so we have highest degree $\lambda_{i} + 1$ in $u_{\lambda}'(z_{i})$.  Moreover, 
\begin{align*}
\prod_{1 \leq i<j \leq n} (1-tz_{i}^{-1}z_{j}^{-1})(z_{i} - tz_{j}) = \prod_{1 \leq i<j \leq n}(z_{i} - tz_{j}^{-1} - tz_{j} + t^{2}z_{i}^{-1})
\end{align*}
has highest degree term $z^{\rho - 1}$.  Thus, the dominating monomial in $K_{\lambda}^{(n)} \cdot \Delta_{BC}$ is $z^{\lambda + \rho}$, so that the dominating monomial in $K_{\lambda}^{(n)}$ is $z^{\lambda}$.  

We now show that the coefficient on $z^{\lambda + \rho}$ in $R_{\lambda}^{(n)} \cdot \Delta_{BC}$ (see (\ref{R-pol}) for the definition of $R_{\lambda}^{(n)}$) is $v_{\lambda}(t)$, so that $K_{\lambda}^{(n)}$ is indeed monic.  Note first that by the above argument the only contributing $w$ are those such that (1) $z_{1}^{\lambda_{1}} \cdots z_{n}^{\lambda_{n}} = z_{w(1)}^{\lambda_{1}} \cdots z_{w(n)}^{\lambda_{n}}$ and (2) $\epsilon_{w}(z_{i}) = 1$ for all $1 \leq i \leq n - m_{0}(\lambda) - m_{1}(\lambda)$; let the set of these special permutations be denoted by $P_{\lambda, n}$.  Now fix $w \in P_{\lambda, n}$, we compute the coefficient on $z_{1}^{\lambda_{1} + n}$.  Using (\ref{Kdelta}) and the arguments of the previous paragraph, one can check that the coefficient is:

\begin{enumerate}
\item If $\lambda_{1} >1$:
\begin{align*}
  t^{\# \{z_{i} \prec_{w} z_{1}\}} 
\end{align*}

\item If $\lambda_{1} = 1$:
\begin{align*}
\begin{cases}
  t^{\# \{z_{i} \prec_{w} z_{1}\}} ,  & \text{if } \epsilon_{w}(z_{1}) = 1 \\
 - t_{0} \cdots t_{3} (t^{2})^{\# \{z_{1} \prec_{w} z_{i}\}} t^{\#\{ z_{i} \prec_{w} z_{1}\}}  ,& \text{if } \epsilon_{w}(z_{1}) = -1
 \end{cases}
\end{align*}

\item If $\lambda_{1} = 0$:
\begin{align*}
\begin{cases}
t^{\# \{z_{i} \prec_{w} z_{1}\}}, & \text{if } \epsilon_{w}(z_{1}) = 1\\
-ab(t^{2})^{\# \{z_{1} \prec_{w} z_{i}\}}t^{\# \{z_{i} \prec_{w} z_{1}\}}, & \text{if } \epsilon_{w}(z_{1}) = -1
\end{cases}
\end{align*}
\end{enumerate}  

(note that we have used the contribution of $(-1)$ factors from $\epsilon(w)$ in $K_{\lambda}^{(n)} \cdot \Delta_{BC}$).  

Now define the following subsets of the variables $z_{1}, \dots, z_{n}$:
\begin{align*}
N_{w, \lambda}^{1} &= \{z_{i} : n-m_{0}(\lambda) - m_{1}(\lambda) < i \leq n-m_{0}(\lambda) \text{ and } \epsilon_{w}(z_{i}) = -1\}\\
N_{w, \lambda}^{0} &= \{ z_{i} : n-m_{0}(\lambda) < i \leq n \text{ and } \epsilon_{w}(z_{i}) = -1 \} \\
N_{w, \lambda} &= N_{w, \lambda}^{1}  + N_{w, \lambda}^{0}.
\end{align*}

Finally, define the following statistics of $w$:
\begin{align*}
n(w) &= |\{ (i,j) : 1 \leq i<j \leq n \text{ and } z_{j}  \prec_{w} z_{i} \}| \\
c_{\lambda}(w) &= |\{ (i,j) : 1 \leq i<j \leq n \text{ and } z_{i} \prec_{w} z_{j} \text{ and } z_{i} \in N_{w, \lambda}\}|.
\end{align*} 

Then by iterating the coefficient argument above, we get that the coefficient on $z^{\lambda + \rho}$ is given by
\begin{equation*}
\sum_{w \in P_{\lambda, n}} t^{n(w)} t^{2c_{\lambda}(w)} (-t_{0} \dots t_{3})^{|N^{1}_{w, \lambda}|} (-ab)^{|N^{0}_{w, \lambda}|}.
\end{equation*}
Since $P_{\lambda, n} = B_{m_{0}(\lambda)}B_{m_{1}(\lambda)} \prod_{i \geq 2} S_{m_{i}(\lambda)}$, it is enough to show the following three cases:
\begin{equation} \label{stat1}
\sum_{w \in S_{m}} t^{n(w)} = \prod_{j=1}^{m} \frac{1-t^{j}}{1-t}
\end{equation}
\begin{equation} \label{stat2}
\sum_{w \in B_{m}} t^{n(w)} t^{2c_{1^{m}}(w) + 2m_{0}(\lambda)} (-t_{0} \dots t_{3})^{\big|N^{1}_{w, 1^{m}} \big|} 
=  \prod_{j=1}^{m} \frac{1-t^{j}}{1-t} (1-t_{0}\cdots t_{3} t^{j-1 + 2m_{0}(\lambda)})
\end{equation}
\begin{equation} \label{stat3}
\sum_{w \in B_{m}} t^{n(w)} t^{2c_{0^{m}}(w)} (-ab)^{\big|N^{0}_{w, 0^{m}}\big|} =  \prod_{j=1}^{m} \frac{1-t^{j}}{1-t} (1-abt^{j-1}).
\end{equation}
To show (\ref{stat1}), we note that the LHS is exactly enumerated by the terms of
\begin{equation*}
(1+t+t^{2} + \cdots + t^{m-1})(1+t+t^{2} + \cdots + t^{m-2}) \cdots (1+t)(1),
\end{equation*}
which is equal to the RHS.  Also refer to \cite[Ch. III, proof of (1.2) and (1.3)]{Mac}.  We now show (\ref{stat2}); (\ref{stat3}) is analogous.  One can verify that the LHS of (\ref{stat2}) is exactly enumerated by the terms of 
\begin{equation} \label{stat2eqn}
\prod_{k=1}^{m}\Big[ \sum_{i=1}^{k} \big( t^{i-1} + t^{i-1}(t^{2})^{m_{0}(\lambda)+k-i}(-t_{0} \cdots t_{3}) \big) \Big].
\end{equation}
But we also have
\begin{multline*}
  \sum_{i=1}^{k} \Big( t^{i-1} + t^{i-1} (t^{2})^{m_{0}(\lambda) + k - i}(-t_{0} \cdots t_{3}) \Big) = \sum_{i=1}^{k} \big(t^{i-1} - t_{0}\cdots t_{3} t^{k+2m_{0}(\lambda)-1} t^{k-i}\big) \\
  = (1-t_{0} \cdots t_{3} t^{k+2m_{0}(\lambda) -1})(1 + t + \cdots t^{k-1}) 
  = (1-t_{0} \cdots t_{3} t^{k+2m_{0}(\lambda) -1}) \frac{1-t^{k}}{1-t};
 \end{multline*}
substituting this into (\ref{stat2eqn}) gives the RHS of (\ref{stat2}) as desired.

Multiplying these functions together for each distinct part $i$ of $\lambda$ (put $m = m_{i}(\lambda)$ in (\ref{stat1}), (\ref{stat2}), and (\ref{stat3}), depending on whether $i \geq 2, i=1, \text{ or } i=0$, respectively), and using (\ref{v}) shows that the coefficient on $z^{\lambda + \rho}$ in $R_{\lambda}^{(n)} \cdot \Delta_{BC}$ is indeed $v_{\lambda}(t)$, as desired.  
\end{proof}

We will now provide a direct proof of Gustafson's formula \cite{G} in the limit $q=0$.  

\begin{theorem} \label{symmeval}
We have the following constant term evaluation in the symmetric case
\begin{multline*}
\int_{T} \tilde \Delta_{K}^{(n)}(z;t;a,b,c,d) dT = \prod_{i=0}^{n-1} \frac{1}{(1-t^{i}ac)(1-t^{i}bc)(1-t^{i}cd)(1-t^{i}ad)(1-t^{i}bd)(1-t^{i}ab)}\\
\times \prod_{j=0}^{n-1} (1-t^{2n-2-j}abcd) \prod_{j=1}^{n} \frac{1-t}{1-t^{j}}.
\end{multline*}
\end{theorem}

\begin{proof}
Note first that by Theorem \ref{symmtriang}, $K_{0^{n}}^{(n)}(z;t;a,b,0,0) = 1$.  So in particular, we have
\begin{multline*}
\int_{T} \tilde \Delta_{K}^{(n)}(z;t;a,b,c,d) dT = \int_{T} K_{0^{n}}^{(n)}(z;t;a,b,0,0) \tilde \Delta_{K}^{(n)}(z;t;a,b,c,d) dT \\
= \frac{1}{v_{0^{n}}(t;a,b,0,0)} \sum_{w \in B_{n}} \int_{T} R_{0^{n},w}^{(n)}(z;t;a,b,0,0)\tilde \Delta_{K}^{(n)}(z;t;a,b,c,d) dT \\
= \frac{2^{n}n!}{v_{0^{n}}(t;a,b,0,0)} \int_{T}R_{0^{n},\text{id}}^{(n)}(z;t;a,b,0,0)\tilde \Delta_{K}^{(n)}(z;t;a,b,c,d) dT, 
\end{multline*}
where the last equality follows by symmetry of the integrand.  But now using (\ref{Kterm}), one notes that
\begin{multline*}
2^{n}n! R_{0^{n},\text{id}}^{(n)}(z;t;a,b,0,0)\tilde \Delta_{K}^{(n)}(z;t;a,b,c,d) \\= \prod_{1 \leq i \leq n} \frac{(1-z_{i}^{2})}{(1-az_{i})(1-bz_{i})(1-cz_{i})(1-dz_{i})(1-cz_{i}^{-1})(1-dz_{i}^{-1})}
 \prod_{1 \leq i<j \leq n} \frac{(1-z_{i}z_{j}^{\pm 1})}{(1-tz_{i}z_{j}^{\pm 1})}.
\end{multline*}
We will denote the right-hand side of the above equation by $\Delta_{K}^{(n)}(z;t;a,b,c,d)$.  

We will now prove that
\begin{equation} \label{nonsymmeqn}
\int_{T} \Delta_{K}^{(n)}(z;t;a,b,c,d) dT = \prod_{i=0}^{n-1} \frac{1}{(1-t^{i}ac)(1-t^{i}bc)(1-t^{i}cd)(1-t^{i}ad)(1-t^{i}bd)} \prod_{j=n-1}^{2n-2} (1-t^{j}abcd).
\end{equation}

For facility of notation, we will put $I_{n}(z;t;a,b;c,d) = \int_{T} \Delta_{K}^{(n)}(z;t;a,b,c,d) dT$.  We will prove (\ref{nonsymmeqn}) through the following two claims.

\textbf{Claim 1:} We have 
\begin{multline} \label{recurrence}
I_{n}(z;t;a,b;c,d) = \frac{c}{(1-ac)(1-bc)(1-dc)(c-d)}I_{n-1}(z;t;a,b;tc,d) \\+ \frac{d}{(1-ad)(1-bd)(1-cd)(d-c)} I_{n-1}(z;t;a,b;c,td),
\end{multline}
with initial conditions $I_{0}(z;t;a,b;c,d) = 1$ and 
\begin{equation*}
I_{1}(z;t;a,b;c,d) = \frac{1-abcd}{(1-ac)(1-bc)(1-cd)(1-ad)(1-bd)}.
\end{equation*}

To prove the first claim, we note that
\begin{multline*}
I_{n}(z;t;a,b;c,d) \\= \int \prod_{1 \leq i \leq n}\frac{z_{i}(1-z_{i}^{2})}{(1-az_{i})(1-bz_{i})(1-cz_{i})(1-dz_{i})(z_{i}-c)(z_{i}-d)} \prod_{1 \leq i<j \leq n} \frac{(z_{j}-z_{i})(1-z_{i}z_{j})}{(z_{j}-tz_{i})(1-tz_{i}z_{j})} \prod_{j=1}^{n} \frac{dz_{j}}{2\pi \sqrt{-1}}.
\end{multline*}
We may now hold the variables $z_{2}, \dots, z_{n}$ fixed and integrate with respect to $z_{1}$.  There are simple poles at $z_{1} = c$ and $z_{1} = d$, so by the Residue Theorem, it will be the sum of residues at these poles.  Consider the residue at $z_{1} = c$: 
\begin{multline*}
\int \prod_{2 \leq i \leq n}\frac{z_{i}(1-z_{i}^{2})}{(1-az_{i})(1-bz_{i})(1-cz_{i})(1-dz_{i})(z_{i}-c)(z_{i}-d)} \prod_{2 \leq i<j \leq n} \frac{(z_{j}-z_{i})(1-z_{i}z_{j})}{(z_{j}-tz_{i})(1-tz_{i}z_{j})}  \\
\times \frac{c}{(1-ac)(1-bc)(1-cd)(c-d)}  \prod_{1<j \leq n} \frac{(z_{j}-c)(1-cz_{j})}{(z_{j}-tc)(1-tcz_{j})} \prod_{j=2}^{n} \frac{dz_{j}}{2\pi \sqrt{-1}} \\
=  C_{1} \int_{T_{n-1}} \prod_{2 \leq i \leq n}\frac{z_{i}(1-z_{i}^{2})}{(1-az_{i})(1-bz_{i})(1-tcz_{i})(1-dz_{i})(z_{i}-tc)(z_{i}-d)} \prod_{2 \leq i<j \leq n} \frac{(z_{j}-z_{i})(1-z_{i}z_{j})}{(z_{j}-tz_{i})(1-tz_{i}z_{j})} dT
\end{multline*}
where $C_{1} = \frac{c}{(1-ac)(1-bc)(1-cd)(c-d)}$.  By renumbering the variables $(z_{2}, \dots, z_{n})$ by $(z_{1}, \dots, z_{n-1})$, one sees that this is exactly $C_{1} I_{n-1}(z;t;a,b;tc,d)$.  An analogous argument applies for the residue at $z_{1} = d$; this produces the second term $C_{2}I_{n-1}(z;t;a,b;c,td)$, where $C_{2} = \frac{d}{(1-ad)(1-bd)(1-cd)(d-c)}$.  

To obtain the result at $n=1$, one uses the above argument in this special case along with some algebraic manipulation.  In particular, the computation of the sum of residues is as follows
\begin{multline*}
 \frac{(1-c^{2})c}{(1-ac)(1-bc)(1-c^{2})(1-dc)(c-d)} + \frac{(1-d^{2})d}{(1-ad)(1-bd)(1-cd)(1-d^{2})(d-c)} \\  = \frac{1}{(1-ac)(1-bc)(1-cd)(1-ad)(1-bd)} \Big[ \frac{c(1-ad)(1-bd)}{c-d} + \frac{d(1-ac)(1-bc)}{d-c} \Big] \\
= \frac{1-abcd}{(1-ac)(1-bc)(1-cd)(1-ad)(1-bd)},
\end{multline*}
as desired.  This proves the first claim.

\textbf{Claim 2:} We have the following solution to (\ref{recurrence})
\begin{equation*}
I_{n}(z;t;a,b;c,d) = \prod_{i=0}^{n-1} \frac{1}{(1-t^{i}ac)(1-t^{i}bc)(1-t^{i}cd)(1-t^{i}ad)(1-t^{i}bd)} \prod_{j=n-1}^{2n-2} (1-t^{j}abcd).
\end{equation*}

We prove the second claim.  One can first check that $n=0,1$ satisfies the initial conditions of (\ref{recurrence}).  Then for $n \geq 2$, we have
\begin{multline*}
\frac{c}{(1-ac)(1-bc)(1-dc)(c-d)}I_{n-1}(z;t;a,b;tc,d) + \frac{d}{(1-ad)(1-bd)(1-cd)(d-c)} I_{n-1}(z;t;a,b;c,td) \\
= \frac{c\displaystyle \prod_{j=n-2}^{2n-4} 1-t^{j+1}abcd}{(1-ac)(1-bc)(1-dc)(c-d)}\prod_{i=0}^{n-2} \frac{1}{(1-t^{i+1}ac)(1-t^{i+1}bc)(1-t^{i+1}cd)(1-t^{i}ad)(1-t^{i}bd)} \\
+ \frac{d \displaystyle \prod_{j=n-2}^{2n-4} 1-t^{j+1}abcd }{(1-ad)(1-bd)(1-cd)(d-c)}\prod_{i=0}^{n-2} \frac{1}{(1-t^{i}ac)(1-t^{i}bc)(1-t^{i+1}cd)(1-t^{i+1}ad)(1-t^{i+1}bd)}  \\
= \Bigg[ \frac{c(1-t^{n-1}ad)(1-t^{n-1}bd)}{c-d} + \frac{d(1-t^{n-1}ac)(1-t^{n-1}bc)}{d-c} \Bigg] \\
\times \prod_{i=0}^{n-1} \frac{1}{(1-t^{i}ac)(1-t^{i}bc)(1-t^{i}cd)(1-t^{i}ad)(1-t^{i}bd)} \prod_{j=n-2}^{2n-4} (1-t^{j+1}abcd) 
\end{multline*}
But now note the following identity for the sum inside the parentheses:
\begin{multline*}
\frac{c(1-t^{n-1}ad)(1-t^{n-1}bd)}{c-d} + \frac{d(1-t^{n-1}ac)(1-t^{n-1}bc)}{d-c}  \\= \frac{c(1-t^{n-1}ad)(1-t^{n-1}bd) -d(1-t^{n-1}ac)(1-t^{n-1}bc) }{c-d} \\ = \frac{c-d+t^{2(n-1)}abcd^{2} - t^{2(n-1)}abc^{2}d}{c-d} = 1-t^{2(n-1)}abcd,
\end{multline*}
so the above finally becomes
\begin{equation*}
\prod_{i=0}^{n-1} \frac{1}{(1-t^{i}ac)(1-t^{i}bc)(1-t^{i}cd)(1-t^{i}ad)(1-t^{i}bd)} \prod_{j=n-1}^{2(n-1)} (1-t^{j}abcd) = I_{n}(z;t;a,b;c,d),
\end{equation*}
which proves (\ref{nonsymmeqn}).

Thus, putting this together we have
\begin{multline*}
\int_{T} \tilde \Delta_{K}^{(n)}(z;t;a,b,c,d) dT =  \frac{1}{v_{0^{n}}(t;a,b,0,0)} \int_{T} \Delta_{K}^{(n)}(z;t;a,b,c,d) dT \\
= \prod_{i=0}^{n-1} \frac{1}{(1-t^{i}ac)(1-t^{i}bc)(1-t^{i}cd)(1-t^{i}ad)(1-t^{i}bd)} \prod_{j=n-1}^{2n-2} (1-t^{j}abcd) \times \prod_{j=1}^{n} \frac{1-t}{1-t^{j}} \prod_{i=1}^{n} \frac{1}{1-abt^{i-1}} \\
=\prod_{i=0}^{n-1} \frac{1}{(1-t^{i}ac)(1-t^{i}bc)(1-t^{i}cd)(1-t^{i}ad)(1-t^{i}bd)(1-t^{i}ab)} \prod_{j=0}^{n-1} (1-t^{2n-2-j}abcd) \prod_{j=1}^{n} \frac{1-t}{1-t^{j}},
\end{multline*}
where we have used Theorem \ref{nonsymmeval} and (\ref{v}).

\end{proof}
We note that the quantity $\Delta_{K}^{(n)}(z;t;a,b,c,d)$ which appears in the proof of Theorem \ref{symmeval} is actually the $q=0$ limit of the nonsymmetric Koornwinder density (see \cite{RV} for example); the nonsymmetric theory is investigated in the next section.

\begin{theorem} \label{symmorthognorm}
The family of polynomials $\{K_{\lambda}^{(n)}(z;t;a,b;t_{0}, \dots, t_{3}) \}_{\lambda}$ satisfy the following orthogonality result:
\begin{align*}
\int_{T} K_{\lambda}^{(n)}(z;t;a,b; t_{0}, \dots, t_{3}) K_{\mu}^{(n)}(z;t;a,b; t_{0}, \dots, t_{3}) \tilde \Delta_{K}^{(n)}(z;t;t_{0}, \dots, t_{3}) dT &= N_{\lambda}(t;t_{0}, \dots, t_{3}) \delta_{\lambda \mu}
\end{align*}
(refer to (\ref{koorndensity}) and (\ref{norm}) for the definitions of $ \tilde \Delta_{K}^{(n)}$ and $N_{\lambda}$, respectively; also see Theorem \ref{symmeval}).
\end{theorem}

\begin{proof}
By symmetry of $\lambda, \mu$, we may restrict to the case where $\lambda \stackrel{\text{lex}}{\geq} \mu$.  We assume $\lambda_{1} > 0$, so we need not consider the case $\lambda = \mu = 0^{n}$; these assumptions hold throughout the proof.  By definition of $K_{\lambda}^{(n)}(z;t;a,b;t_{0}, \dots, t_{3})$ as a sum over $B_{n}$, the above integral is equal to
\begin{equation*}
\sum_{w,\rho \in B_{n}} \int_{T} K_{\lambda, w}^{(n)}(z;t;a,b;t_{0}, \dots, t_{3}) K_{\mu, \rho}^{(n)}(z;t;a,b;t_{0}, \dots, t_{3}) \tilde \Delta_{K}^{(n)}(z;t;t_{0}, \dots, t_{3}) dT.
\end{equation*}
Consider an arbitrary term in this sum over $B_{n} \times B_{n}$ indexed by $(w, \rho)$.  Note that using a change of variables in the integral and inverting variables (which preserves the integral), we may assume $w$ is the identity permutation, and all sign choices are $1$ (and $\rho$ is arbitrary).  That is, we have:
\begin{multline*}
\int_{T} K_{\lambda}(z_{1}, \dots, z_{n};t;a,b; t_{0}, \dots, t_{3}) K_{\mu}(z_{1}, \dots, z_{n};t;a,b; t_{0}, \dots, t_{3}) \tilde \Delta_{K}^{(n)}(z;t;t_{0}, \dots, t_{3}) dT  \\
= 2^{n}n! \sum_{\rho \in B_{n}} \int_{T} K_{\lambda, \text{id}}^{(n)}(z;t;a,b;t_{0}, \dots, t_{3}) K_{\mu, \rho}^{(n)}(z;t;a,b;t_{0}, \dots, t_{3}) \tilde \Delta_{K}^{(n)}(z;t;t_{0}, \dots, t_{3}) dT \\
= 2^{n}n! \frac{1}{v_{\lambda}(t)v_{\mu}(t)}\sum_{\rho \in B_{n}} \int_{T} R_{\lambda, \text{id}}^{(n)}(z;t;a,b;t_{0}, \dots, t_{3}) R_{\mu, \rho}^{(n)}(z;t;a,b;t_{0}, \dots, t_{3}) \tilde \Delta_{K}^{(n)}(z;t;t_{0}, \dots, t_{3}) dT,
\end{multline*}
where $R_{\lambda}^{(n)}$ is as defined in (\ref{R-pol}).  

We study an arbitrary term in this sum.  In particular, we give an iterative formula that shows that each of these terms vanishes unless $\lambda = \mu$.  
\begin{claim} Fix an arbitrary $\rho \in B_{n}$ and let $\rho(i) = 1$ for some $1 \leq i \leq n$.  Then we have the following formula:
\begin{multline*}
2^{n}n!\int_{T} R_{\lambda, \text{id}}^{(n)}(z;t;a,b;t_{0}, \dots, t_{3}) R_{\mu, \rho}^{(n)}(z;t;a,b;t_{0}, \dots, t_{3}) \tilde \Delta_{K}^{(n)} dT \\
=\begin{cases} 
t^{i-1} 2^{n-1}(n-1)! \int R_{\widehat{\lambda},  \widehat{\text{id}}}^{(n-1)} R_{\widehat{\mu}, \widehat{\rho}}^{(n-1)} \tilde \Delta_{K}^{(n-1)} dT & \text{if $\mu_{i} = \lambda_{1}$ and $\epsilon_{\rho}(z_{1}) = -1$,}\\
t^{i-1}(t^{2})^{m_{0}(\mu)+m_{1}(\mu)-i}(-t_{0} \cdots t_{3})  2^{n-1}(n-1)!\int R_{\widehat{\lambda}, \widehat{\text{id}}}^{(n-1)} R_{\widehat{\mu}, \widehat{\rho}}^{(n-1)} \tilde \Delta_{K}^{(n-1)} dT & \textit{if $\mu_{i} = \lambda_{1}=1$} \\ &\text{ and $\epsilon_{\rho}(z_{1}) = 1$}, \\ 
0 & \textit{otherwise.}\\
\end{cases}
\end{multline*}
where $\widehat{\lambda}$ and $\widehat{\mu}$ are the partitions $\lambda$ and $\mu$ with parts $\lambda_{1}$ and $\mu_{i}$ deleted (respectively), and $\widehat{\text{id}}$ and $\widehat{\rho}$ are the permutations $\text{id}$ and $\rho$ with $z_{1}$ deleted (respectively) and signs preserved.
\end{claim}

To prove the claim, we integrate with respect to $z_{1}$ in the iterated integral, using the definition of $R^{(n)}_{\lambda, \text{id}}, R^{(n)}_{\mu, \rho}$ and $\tilde \Delta_{K}^{(n)}$.  

First suppose $\mu_{i} > 0$.  The univariate terms in $z_{1}$ are:
\begin{align*}
& z_{1}^{\lambda_{1}} \frac{(1-t_{0}z_{1}^{-1}) \cdots (1-t_{3}z_{1}^{-1})}{(1-z_{1}^{-2})} z_{1}^{\mu_{i}} \frac{(1-t_{0}z_{1}^{-1}) \cdots (1-t_{3}z_{1}^{-1})}{(1-z_{1}^{-2})} \frac{(1-z_{1}^{\pm 2})}{(1-t_{0}z_{1}^{\pm 1}) \cdots (1-t_{3}z_{1}^{\pm 1})} \\
&= z_{1}^{\lambda_{1} + \mu_{i}} \frac{(-z_{1}^{2})(1-t_{0}z_{1}^{-1}) \cdots (1-t_{3}z_{1}^{-1})}{(1-t_{0}z_{1}) \cdots (1-t_{3}z_{1})}
\end{align*}
if $\epsilon_{\rho}(z_{1}) = 1$, and
\begin{align*}
& z_{1}^{\lambda_{1}} \frac{(1-t_{0}z_{1}^{-1}) \cdots (1-t_{3}z_{1}^{-1})}{(1-z_{1}^{-2})} z_{1}^{-\mu_{i}} \frac{(1-t_{0}z_{1}) \cdots (1-t_{3}z_{1})}{(1-z_{1}^{2})} \frac{(1-z_{1}^{\pm 2})}{(1-t_{0}z_{1}^{\pm 1}) \cdots (1-t_{3}z_{1}^{\pm 1})} \\
&= z_{1}^{\lambda_{1} - \mu_{i}} 
\end{align*}
if $\epsilon_{\rho}(z_{1}) = -1$.

Now suppose $\mu_{i} = 0$.  The univariate terms in $z_{1}$ are:
\begin{align*}
& z_{1}^{\lambda_{1}} \frac{(1-t_{0}z_{1}^{-1}) \cdots (1-t_{3}z_{1}^{-1})}{(1-z_{1}^{-2})} \frac{(1-az_{1}^{-1})(1-bz_{1}^{-1})}{(1-z_{1}^{-2})} \frac{(1-z_{1}^{\pm 2})}{(1-t_{0}z_{1}^{\pm 1}) \cdots (1-t_{3}z_{1}^{\pm 1})} \\
&= z_{1}^{\lambda_{1}} \frac{(-z_{1}^{2})(1-az_{1}^{-1})(1-bz_{1}^{-1})}{(1-t_{0}z_{1}) \cdots (1-t_{3}z_{1})}
\end{align*}
if $\epsilon_{\rho}(z_{1}) = 1$, and
\begin{align*}
& z_{1}^{\lambda_{1}} \frac{(1-t_{0}z_{1}^{-1}) \cdots (1-t_{3}z_{1}^{-1})}{(1-z_{1}^{-2})} \frac{(1-az_{1})(1-bz_{1})}{(1-z_{1}^{2})} \frac{(1-z_{1}^{\pm 2})}{(1-t_{0}z_{1}^{\pm 1}) \cdots (1-t_{3}z_{1}^{\pm 1})} \\
&= z_{1}^{\lambda_{1}} \frac{(1-az_{1})(1-bz_{1})}{(1-t_{0}z_{1}) \cdots (1-t_{3}z_{1})}
\end{align*}
if $\epsilon_{\rho}(z_{1}) = -1$.  

Notice that for the cross terms in $z_{1}$ (those involving $z_{j}$ for $j \neq 1$), we have
\begin{align*}
\prod_{j>1} \frac{1-tz_{1}^{-1}z_{j}^{-1}}{1-z_{1}^{-1}z_{j}^{-1}} \frac{1-tz_{1}^{-1}z_{j}}{1-z_{1}^{-1}z_{j}} \times \prod_{j>1} \frac{1-z_{1}^{\pm 1}z_{j}^{\pm 1}}{1-tz_{1}^{\pm 1}z_{j}^{\pm 1}}
\end{align*}
from the corresponding terms in $z_{1}$ of $R_{\lambda, \text{id}}$ and the density.  Combining this with the cross terms of $R_{\mu, \rho}$ in $z_{1}$ (and taking into account the various sign possibilities for $\rho$), we obtain
\begin{align*}
\prod_{\substack{z_{i} \prec_{\rho} z_{1} \\ \text{ sign } 1 \text{ for } z_{i}}} \frac{t-z_{1}z_{i}}{1-tz_{1}z_{i}} \prod_{\substack{z_{i} \prec_{\rho} z_{1} \\ \text{ sign } -1 \text{ for } z_{i}}} \frac{t-z_{1}z_{i}^{-1}}{1-tz_{1}z_{i}^{-1}} \prod_{z_{1} \prec_{\rho} z_{j}} \frac{(t-z_{1}z_{j}^{-1})(t-z_{1}z_{j})}{(1-tz_{1}z_{j}^{-1})(1-tz_{1}z_{j})}
\end{align*}
if $\epsilon_{\rho}(z_{1}) = 1$, and
\begin{align*}
\prod_{\substack{z_{i} \prec_{\rho} z_{1} \\ \text{ sign } 1 \text{ for } z_{i}}} \frac{t-z_{1}z_{i}}{1-tz_{1}z_{i}} \prod_{\substack{z_{i} \prec_{\rho} z_{1} \\ \text{ sign } -1 \text{ for } z_{i}}} \frac{t-z_{1}z_{i}^{-1}}{1-tz_{1}z_{i}^{-1}}
\end{align*}
if $\epsilon_{\rho}(z_{1}) = -1$.  

Thus, combining these computations, the integral in $z_{1}$ is:
\begin{multline*}
\begin{cases}
\int_{T_{1}} z_{1}^{\lambda_{1} + \mu_{i}} \frac{(-z_{1}^{2})(1-t_{0}z_{1}^{-1}) \cdots (1-t_{3}z_{1}^{-1})}{(1-t_{0}z_{1}) \cdots (1-t_{3}z_{1})} \cdot \\ \displaystyle \prod_{\substack{ z_{k} \prec_{\rho} z_{1} \\ \epsilon_{\rho}(z_{k}) = 1}} \frac{t-z_{1}z_{k}}{1-tz_{1}z_{k}} \prod_{\substack{z_{k} \prec_{\rho} z_{1} \\ \epsilon_{\rho}(z_{k})=-1}} \frac{t-z_{1}z_{k}^{-1}}{1-tz_{1}z_{k}^{-1}} \prod_{z_{1} \prec_{\rho} z_{j}} \frac{(t-z_{1}z_{j}^{-1})(t-z_{1}z_{j})}{(1-tz_{1}z_{j}^{-1})(1-tz_{1}z_{j})} dT & \text{if $\mu_{i} > 0$ and $\epsilon_{\rho}(z_{1}) = 1$,}
\\
 \int_{T_{1}} z_{1}^{\lambda_{1}} \frac{(-z_{1}^{2})(1-az_{1}^{-1})(1-bz_{1}^{-1})}{(1-t_{0}z_{1}) \cdots (1-t_{3}z_{1})} \cdot \\ \displaystyle \prod_{\substack{z_{k} \prec_{\rho} z_{1} \\ \epsilon_{\rho}(z_{k}) = 1}} \frac{t-z_{1}z_{k}}{1-tz_{1}z_{k}} \prod_{\substack{z_{k} \prec_{\rho} z_{1} \\ \epsilon_{\rho}(z_{k}) = -1}} \frac{t-z_{1}z_{k}^{-1}}{1-tz_{1}z_{k}^{-1}} \prod_{z_{1} \prec_{\rho} z_{j}} \frac{(t-z_{1}z_{j}^{-1})(t-z_{1}z_{j})}{(1-tz_{1}z_{j}^{-1})(1-tz_{1}z_{j})} dT & \text{if $\mu_{i} = 0$ and $\epsilon_{\rho}(z_{1}) = 1$,} 
 \\
 \int_{T_{1}} z_{1}^{\lambda_{1} - \mu_{i}} \displaystyle \prod_{\substack{z_{k} \prec_{\rho} z_{1} \\ \epsilon_{\rho}(z_{k}) = 1 }} \frac{t-z_{1}z_{k}}{1-tz_{1}z_{k}} \prod_{\substack{z_{k} \prec_{\rho} z_{1} \\ \epsilon_{\rho}(z_{k})=-1}} \frac{t-z_{1}z_{k}^{-1}}{1-tz_{1}z_{k}^{-1}} dT & \text{if $\mu_{i} > 0$ and $\epsilon_{\rho}(z_{1}) = -1$,}
\\
\int_{T_{1}} z_{1}^{\lambda_{1}} \frac{(1-az_{1})(1-bz_{1})}{(1-t_{0}z_{1}) \cdots (1-t_{3}z_{1})}\displaystyle \prod_{\substack{z_{k} \prec_{\rho} z_{1} \\ \epsilon_{\rho}(z_{k}) = 1}} \frac{t-z_{1}z_{k}}{1-tz_{1}z_{k}} \prod_{\substack{z_{k} \prec_{\rho} z_{1} \\ \epsilon_{\rho}(z_{k}) = -1}} \frac{t-z_{1}z_{k}^{-1}}{1-tz_{1}z_{k}^{-1}} dT&\text{if $\mu_{i} = 0$ and $\epsilon_{\rho}(z_{1}) = -1$.}
\end{cases}
\end{multline*}
In particular, the first integral vanishes unless $\lambda_{1} = \mu_{i} = 1$; the second integral always vanishes; the third integral vanishes unless $\lambda_{1} = \mu_{i}$; the fourth integral always vanishes.  Thus, we obtain the vanishing conditions of the claim.  To obtain the nonzero values, one can use the residue theorem and evaluate at the simple pole $z_{1} = 0$ in the cases $\lambda_{1} = \mu_{i} = 1$ and $\lambda_{1} = \mu_{i}$.  Finally, we combine with the original integrand involving terms in $z_{2}, \dots, z_{n}$ to obtain the result of the claim.

Note that in particular the claim implies that if $\lambda \neq \mu$, each term vanishes and consequently the total integral is zero.  This proves the vanishing part of the orthogonality statement.  

Next, we compute the norm when $\lambda = \mu$.  The claim shows that only certain $\rho \in B_{n}$ give nonvanishing term integrals.  Such permutations must satisfy
\begin{align*}
z_{1}^{\lambda_{1}} \cdots z_{n}^{\lambda_{n}} z_{\rho(1)}^{-\lambda_{1}} \cdots z_{\rho(n)}^{-\lambda_{n}} = 1
\end{align*}
and $\epsilon_{\rho}(z_{i}) = -1$ for all $1 \leq i \leq n - m_{0}(\lambda) - m_{1}(\lambda)$.  For simplicity of notation, define $B_{\lambda,n}$ to be the set of such permutations $\rho \in B_{n}$.  Then we have:
\begin{multline*}
\int_{T} K_{\lambda}^{(n)}(z;t;a,b;t_{0}, \dots, t_{3}) K_{\lambda}^{(n)}(z;t;a,b;t_{0}, \dots, t_{3}) \tilde \Delta_{K}^{(n)} dT = \frac{2^{n}n!}{v_{\lambda}(t)^{2}} \sum_{\rho \in B_{n}} \int_{T} R_{\lambda, \text{id}}^{(n)} R_{\lambda, \rho}^{(n)} \tilde \Delta_{K}^{(n)} dT \\
=\frac{2^{n}n!}{v_{\lambda}(t)^{2}} \sum_{\rho \in B_{\lambda,n}} \int_{T} R_{\lambda, \text{id}}^{(n)} R_{\lambda, \rho}^{(n)} \tilde \Delta_{K}^{(n)} dT,
\end{multline*}
since only these permutations give nonvanishing terms.

Then, using the formula of the Claim, we have
\begin{multline*}
2^{n}n! \sum_{\rho \in B_{\lambda,n}} \int_{T} R_{\lambda, \text{id}}^{(n)} R_{\lambda, \rho}^{(n)} \tilde \Delta_{K}^{(n)} dT \\
= \begin{cases} C_{1} \times 2^{n-m_{\lambda_{1}}(\lambda)} (n-m_{\lambda_{1}}(\lambda))! \displaystyle \sum_{\rho \in B_{\tilde \lambda, n-m_{\lambda_{1}}(\lambda)}} \int_{T} R_{\tilde \lambda, \text{id}}^{(n-m_{\lambda_{1}}(\lambda))} R_{\tilde \lambda, \rho}^{(n-m_{\lambda_{1}}(\lambda))} \tilde \Delta_{K}^{(n-m_{\lambda_{1}}(\lambda))}dT &\text{ if $\lambda_{1}>1$, }\\
C_{2} \times 2^{n-m_{\lambda_{1}}(\lambda)} (n-m_{\lambda_{1}}(\lambda))! \displaystyle \sum_{\rho \in B_{\tilde \lambda, n-m_{\lambda_{1}}(\lambda)}} \int_{T} R_{\tilde \lambda, \text{id}}^{(n-m_{\lambda_{1}}(\lambda))} R_{\tilde \lambda, \rho}^{(n-m_{\lambda_{1}}(\lambda))} \tilde \Delta_{K}^{(n-m_{\lambda_{1}}(\lambda))}dT &\text{ if $\lambda_{1} = 1$,}
\end{cases}
\end{multline*}
where 
\begin{align*}
C_{1} &= \prod_{k=1}^{m_{\lambda_{1}}(\lambda)} \Big( \sum_{i=1}^{k} t^{i-1}\Big) \\
C_{2} &= \prod_{k=1}^{m_{1}(\lambda)} \Big[ \sum_{i=1}^{k} \big(t^{i-1} + t^{i-1}(t^{2})^{m_{0}(\lambda)+k-i}(-t_{0}t_{1}t_{2}t_{3})\big)\Big]
\end{align*}
and $\tilde \lambda$ is the partition $\lambda$ with all $m_{\lambda_{1}}(\lambda)$ occurrences of $\lambda_{1}$ deleted.  Iterating this argument gives that
\begin{multline*}
2^{n}n! \sum_{\rho \in B_{\lambda,n}} \int_{T} R_{\lambda, \text{id}}^{(n)} R_{\lambda, \rho}^{(n)} \tilde \Delta_{K}^{(n)} dT \\
= \Bigg( \prod_{j>1} \prod_{k=1}^{m_{j}(\lambda)}\Big( \sum_{i=1}^{k} t^{i-1}
 \Big) \Bigg) \Bigg( \prod_{k=1}^{m_{1}(\lambda)} \sum_{i=1}^{k} \Big( t^{i-1} + t^{i-1} (t^{2})^{m_{0}(\lambda) + k - i}(-t_{0} \cdots t_{3}) \Big) \Bigg)\\ 
 \times 2^{m_{0}(\lambda)} m_{0}(\lambda)! \sum_{\rho \in B_{m_{0}(\lambda)}} \int_{T}R_{0^{m_{0}(\lambda)},\text{id}}^{(m_{0}(\lambda))} R_{0^{m_{0}(\lambda)}, \rho}^{(m_{0}(\lambda))} \tilde \Delta_{K}^{(m_{0}(\lambda))} dT;
\end{multline*}
note that the expression on the final line is exactly $\int_{T} {R_{0^{m_{0}(\lambda)}}^{(m_{0}(\lambda))}}^{2} \tilde \Delta_{K}^{(m_{0}(\lambda))}dT$.  

Thus,
\begin{multline*}
\frac{2^{n}n!}{v_{\lambda}(t)^{2}} \sum_{\rho \in B_{\lambda,n}} \int_{T} R_{\lambda, \text{id}}^{(n)} R_{\lambda, \rho}^{(n)} \tilde \Delta_{K}^{(n)} dT\\
 = \frac{1}{v_{\lambda+}(t)^{2}}\Bigg( \prod_{j>1} \prod_{k=1}^{m_{j}(\lambda)}\Big( \sum_{i=1}^{k} t^{i-1}
 \Big) \Bigg) \Bigg( \prod_{k=1}^{m_{1}(\lambda)} \sum_{i=1}^{k} \Big( t^{i-1} + t^{i-1} (t^{2})^{m_{0}(\lambda) + k - i}(-t_{0} \cdots t_{3}) \Big) \Bigg) \\
\times \frac{1}{v_{0^{m_{0}(\lambda)}}(t)^{2}} \int_{T} {R_{0^{m_{0}(\lambda)}}^{(m_{0}(\lambda))}}^{2} \tilde \Delta_{K}^{(m_{0}(\lambda))} dT,
 \end{multline*}
since by (\ref{v}) and (\ref{vplus}) we have $v_{\lambda+}(t) \cdot v_{0^{m_{0}(\lambda)}}(t) = v_{\lambda}(t)$.
 Now using
 \begin{align*}
 \prod_{k=1}^{m_{j}(\lambda)} \Big( \sum_{i=1}^{k} t^{i-1} \Big) &= \prod_{k=1}^{m_{j}(\lambda)} \frac{1-t^{k}}{1-t}
 \end{align*}
 and
 \begin{align*}
  \sum_{i=1}^{k} \Big( t^{i-1} + t^{i-1} (t^{2})^{m_{0}(\lambda) + k - i}(-t_{0} \cdots t_{3}) \Big) &= \sum_{i=1}^{k} \big(t^{i-1} - t_{0}\cdots t_{3} t^{k+2m_{0}(\lambda)-1} t^{k-i}\big) \\
  &= (1-t_{0} \cdots t_{3} t^{k+2m_{0}(\lambda) -1})(1 + t + \cdots t^{k-1}) \\
  &= (1-t_{0} \cdots t_{3} t^{k+2m_{0}(\lambda) -1}) \frac{1-t^{k}}{1-t},
 \end{align*}
the above expression can be simplified to 
\begin{multline*}
 \frac{1}{v_{\lambda+}(t)^{2}}\Big(\prod_{j \geq 1} \prod_{k=1}^{m_{j}(\lambda)} \frac{1-t^{k}}{1-t} \Big) \prod_{k=1}^{m_{1}(\lambda)} (1-t_{0} \cdots t_{3} t^{k+2m_{0}(\lambda)-1}) \int_{T} {K_{0^{m_{0}(\lambda)}}^{(m_{0}(\lambda))}}^{2} \tilde \Delta_{K}^{(m_{0}(\lambda))} dT 
= \frac{1}{v_{\lambda+}(t)} \int_{T} \tilde \Delta_{K}^{(m_{0}(\lambda))} dT\\
= N_{\lambda}(t;t_{0}, \dots, t_{3})
\end{multline*}
since $K_{0^{m_{0}(\lambda)}}^{(m_{0}(\lambda))} = 1$, by Theorem \ref{symmtriang}.  Note that, by Theorem \ref{symmeval}, there is an explicit evaluation for this norm.  
\end{proof}

\subsection{Application}
In this section, we use the closed formula (\ref{K-pol}) for the Koornwinder polynomials at $q=0$ to prove a result from \cite{RV} in this special case.  The idea is a type $BC$ adaptation of that used in \cite{VV}: we use the structure of $K_{\lambda}^{(n)}$ as a sum over the Weyl group and the symmetry of the integral to restrict to one particular term.  We obtain an explicit formula for the integral of this particular term by integrating with respect to one variable (holding the others fixed) and then proceed by induction.

\begin{theorem} \cite[Theorem 4.10]{RV} For partitions $\lambda$ with $l(\lambda) \leq n$, the integral 
\begin{equation*}
\int_{T} K_{\lambda}(z_{1}, \dots, z_{n} ; t^{2}; a,b;a,b,ta,tb) \tilde \Delta_{K}^{(n)}(z;t; \pm \sqrt{t},a,b) dT .
\end{equation*}
vanishes if $\lambda$ is not an even partition.  If $\lambda$ is an even partition, the integral is equal to
\begin{equation*}
\frac{(\sqrt{t})^{|\lambda|}}{(1+t)^{l(\lambda)}} \frac{N_{\lambda}(t; \pm \sqrt{t},a,b) v_{\lambda+}(t; \pm \sqrt{t},a,b)}{v_{\lambda+}(t^{2};a,b,ta,tb)}.\end{equation*}
\end{theorem}
\begin{proof}
We have
\begin{multline*}
\int_{T} K_{\lambda}(z_{1}, \dots, z_{n} ; t^{2}; a,b;a,b,ta,tb) \tilde \Delta_{K}^{(n)}(z;t; \pm \sqrt{t},a,b) dT .
\\
= \frac{1}{v_{\lambda}(t^{2};a,b;a,b,ta,tb)} \sum_{w \in B_{n}} \int_{T} R_{\lambda,w}^{(n)}(z;t^{2};a,b;a,b,ta,tb)  \tilde \Delta_{K}^{(n)}(z;t; \pm \sqrt{t},a,b) dT  \\
= \frac{2^{n}n!}{v_{\lambda}(t^{2};a,b;a,b,ta,tb)} \int_{T} R_{\lambda,\text{id}}^{(n)}(z;t^{2};a,b;a,b,ta,tb)  \tilde \Delta_{K}^{(n)}(z;t; \pm \sqrt{t},a,b) dT,
\end{multline*}
where in the last equation we have used the symmetry of the integral.  We assume $\lambda_{1} > 0$ so that $\lambda \neq 0^{n}$.  Next, we restrict to terms involving $z_{1}$ in the integrand, and integrate with respect to $z_{1}$.  Doing this computation gives the following:
\begin{multline*}
\int_{T_{1}} z_{1}^{\lambda_{1}} \frac{(1-az_{1}^{-1})(1-bz_{1}^{-1})(1-taz_{1}^{-1})(1-tbz_{1}^{-1})}{(1-z_{1}^{-2})} \frac{(1-z_{1}^{\pm1})}{(1+\sqrt{t}z_{1}^{\pm 1})(1-\sqrt{t}z_{1}^{\pm 1})(1-az_{1}^{\pm 1})(1-bz_{1}^{\pm 1})} \\
\times \prod_{j>1} \frac{(1-t^{2}z_{1}^{-1}z_{j})(1-t^{2}z_{1}^{-1}z_{j}^{-1})}{(1-z_{1}^{-1}z_{j})(1-z_{1}^{-1}z_{j}^{-1})} \prod_{j>1} \frac{(1-z_{1}^{\pm 1}z_{j}^{\pm 1})}{(1-tz_{1}^{\pm 1}z_{j}^{\pm 1})} dT \\
= \frac{1}{2\pi i} \int_{C} z_{1}^{\lambda_{1} - 1} \frac{(z_{1}-ta)(z_{1}-tb)(1-z_{1}^{2})}{(1-tz_{1}^{2})(z_{1}+\sqrt{t})(z_{1}-\sqrt{t})(1-az_{1})(1-bz_{1})} \\
\times \prod_{j>1} \frac{(z_{1}-t^{2}z_{j})(z_{1}-t^{2}z_{j}^{-1})(1-z_{1}z_{j})(1-z_{1}z_{j}^{-1})}{(z_{1}-tz_{j})(z_{1}-tz_{j}^{-1})(1-tz_{1}z_{j})(1-tz_{1}z_{j}^{-1})} dz_{1}
\end{multline*}
Note that this integral has poles at $z_{1} = \pm \sqrt{t}$ and $z_{1} = tz_{j}, tz_{j}^{-1}$ for each $j>1$.  

We first compute the residue at $z_{1} = \sqrt{t}$:
\begin{multline*}
(\sqrt{t})^{\lambda_{1} - 1} \frac{(\sqrt{t}-ta)(\sqrt{t}-tb)(1-t)}{(1-t^{2})2\sqrt{t}(1-a\sqrt{t})(1-b\sqrt{t})} \prod_{j>1} \frac{(\sqrt{t}-t^{2}z_{j})(\sqrt{t}-t^{2}z_{j}^{-1})(1-\sqrt{t}z_{j})(1-\sqrt{t}z_{j}^{-1})}{(\sqrt{t}-tz_{j})(\sqrt{t}-tz_{j}^{-1})(1-t\sqrt{t}z_{j})(1-t\sqrt{t}z_{j}^{-1})} \\
=(\sqrt{t})^{\lambda_{1}} \frac{1}{2(1+t)} \prod_{j>1} \frac{(1-t\sqrt{t}z_{j})(1-t\sqrt{t}z_{j}^{-1})(1-\sqrt{t}z_{j})(1-\sqrt{t}z_{j}^{-1})}{(1-\sqrt{t}z_{j})(1-\sqrt{t}z_{j}^{-1})(1-t\sqrt{t}z_{j})(1-t\sqrt{t}z_{j}^{-1})} = \frac{(\sqrt{t})^{\lambda_{1}}}{2(1+t)}
\end{multline*}

Similarly, we can compute the residue at $z_{1} = -\sqrt{t}$:
\begin{multline*}
(-\sqrt{t})^{\lambda_{1}-1} \frac{(-\sqrt{t}-ta)(-\sqrt{t}-tb)(1-t)}{(1-t^{2})(-2\sqrt{t})(1+a\sqrt{t})(1+b\sqrt{t})} \prod_{j>1} \frac{(-\sqrt{t}-t^{2}z_{j})(-\sqrt{t}-t^{2}z_{j}^{-1})(1+\sqrt{t}z_{j})(1+\sqrt{t}z_{j}^{-1})}{(-\sqrt{t}-tz_{j})(-\sqrt{t}-tz_{j}^{-1})(1+t\sqrt{t}z_{j})(1+t\sqrt{t}z_{j}^{-1})} \\
= (-\sqrt{t})^{\lambda_{1}} \frac{1}{2(1+t)} \prod_{j>1} \frac{(1+t\sqrt{t}z_{j})(1+t\sqrt{t}z_{j}^{-1})(1+\sqrt{t}z_{j})(1+\sqrt{t}z_{j}^{-1})}{(1+\sqrt{t}z_{j})(1+\sqrt{t}z_{j}^{-1})(1+t\sqrt{t}z_{j})(1+t\sqrt{t}z_{j}^{-1})} = \frac{(-\sqrt{t})^{\lambda_{1}}}{2(1+t)}
\end{multline*}

The residues at $tz_{j}, tz_{j}^{-1}$ can be computed in a similar manner.  One can then combine these residues (at $tz_{j}, tz_{j}^{-1}$) with the terms from the original integrand and integrate with respect to $z_{j}$.  Some computations show the resulting integral is zero; the argument is similar that used in \cite[Theorem 23]{VV}.  

Finally, we add the residues at $z_{1} = \pm \sqrt{t}$ to get
\begin{equation*}
\frac{(\sqrt{t})^{\lambda_{1}}}{2(1+t)} + \frac{(-\sqrt{t})^{\lambda_{1}}}{2(1+t)} = \begin{cases} 
\frac{(\sqrt{t})^{\lambda_{1}}}{(1+t)} ,& \text{if $\lambda_{1}$ is even} \\
0, & \text{if $\lambda_{1}$ is odd.}
\end{cases}
\end{equation*}

Thus,
\begin{multline*}
2^{n}n! \int_{T} R_{\lambda,\text{id}}^{(n)}(z;t^{2};a,b;a,b,ta,tb)  \tilde \Delta_{K}^{(n)}(z;t; \pm \sqrt{t},a,b) dT \\
= \begin{cases} \frac{(\sqrt{t})^{\lambda_{1}}}{(1+t)} 2^{n-1}(n-1)!  \int_{T} R_{\widehat{\lambda},\widehat{\text{id}}}^{(n-1)}(z;t^{2};a,b;a,b,ta,tb)  \tilde \Delta_{K}^{(n-1)}(z;t; \pm \sqrt{t},a,b) dT, & \text{ if $\lambda_{1}$ is even} \\
0, & \text{otherwise,}
\end{cases}
\end{multline*}
where $\widehat{\lambda}$ is the partition $\lambda$ with the part $\lambda_{1}$ deleted, and $\widehat{\text{id}}$ is the permutation $\text{id}$ with $z_{1}$ deleted and signs preserved.

Consequently, the entire integral vanishes if any part is odd and if $\lambda$ is even, it is equal to
\begin{multline*}
\frac{2^{n}n!}{v_{\lambda}(t^{2};a,b;a,b,ta,tb)} \int_{T} R_{\lambda,\text{id}}^{(n)}(z;t^{2};a,b;a,b,ta,tb)  \tilde \Delta_{K}^{(n)}(z;t; \pm \sqrt{t},a,b) dT \\
= \frac{2^{n-l(\lambda)}(n-l(\lambda)!}{v_{\lambda+}(t^{2};a,b,ta,tb)v_{0^{n-l(\lambda)}}(t^{2};a,b;a,b,ta,tb)} \frac{(\sqrt{t})^{|\lambda|}}{(1+t)^{l(\lambda)}}\\
\times \int_{T} R_{0^{n-l(\lambda)},\text{id}}^{(n-l(\lambda))}(z;t^{2};a,b;a,b,ta,tb)  \tilde \Delta_{K}^{(n-l(\lambda))}(z;t; \pm \sqrt{t},a,b) dT, 
\end{multline*}
where, by abuse of notation in the last line, we use $\text{id}$ to denote the identity element in $B_{n-l(\lambda)}$.  By (\ref{R-pol}), the last line is equal to

\begin{multline*}
\frac{2^{n-l(\lambda)}(n-l(\lambda)!}{v_{\lambda+}(t^{2};a,b,ta,tb)} \frac{(\sqrt{t})^{|\lambda|}}{(1+t)^{l(\lambda)}} \int_{T} K_{0^{n-l(\lambda)},\text{id}}^{(n-l(\lambda))}(z;t^{2};a,b;a,b,ta,tb)  \tilde \Delta_{K}^{(n-l(\lambda))}(z;t; \pm \sqrt{t},a,b) dT \\
= \frac{1}{v_{\lambda+}(t^{2};a,b,ta,tb)} \frac{(\sqrt{t})^{|\lambda|}}{(1+t)^{l(\lambda)}} \int_{T} K_{0^{n-l(\lambda)}}^{(n-l(\lambda))}(z;t^{2};a,b;a,b,ta,tb)  \tilde \Delta_{K}^{(n-l(\lambda))}(z;t; \pm \sqrt{t},a,b) dT \\
= \frac{1}{v_{\lambda+}(t^{2};a,b,ta,tb)} \frac{(\sqrt{t})^{|\lambda|}}{(1+t)^{l(\lambda)}} \int_{T}  \tilde \Delta_{K}^{(n-l(\lambda))}(z;t; \pm \sqrt{t},a,b) dT \\
= \frac{(\sqrt{t})^{|\lambda|}}{(1+t)^{l(\lambda)}} \frac{N_{\lambda}(t; \pm \sqrt{t},a,b) v_{\lambda+}(t; \pm \sqrt{t},a,b)}{v_{\lambda+}(t^{2};a,b,ta,tb)},
\end{multline*}
since $K_{0^{l}}^{(l)}(z;t;a,b;t_{0}, \dots, t_{3}) = 1$ by Theorem \ref{symmtriang} and $n-l(\lambda) = m_{0}(\lambda)$.  

\end{proof}

\section{Nonsymmetric Hall-Littlewood polynomials of type BC}

\subsection{Background and Notation}
We first introduce the affine Hecke algebra of type $BC$, a crucial object in the study of nonsymmetric Koornwinder polynomials.  We retain the notation on partitions, compositions and orderings of Section \ref{notation}.  

\begin{definition} (see \cite{RV})
The \textit{affine Hecke algebra $H$ of type $BC$} is defined to be the $\mathbb{C}(q,t,a,b,c,d)$ algebra with generators $T_{0}, T_{1}, \dots, T_{n}$ ($n>1$), subject to the following braid relations
\begin{align*}
T_{i}T_{j} &= T_{j}T_{i}, \; \; \; |i-j| \geq 2, \\
T_{i}T_{j}T_{i} &= T_{j}T_{i}T_{j}, \; \; \; |i-j| = 1, i,j \neq 0,n, \\
T_{i}T_{i+1}T_{i}T_{i+1} &= T_{i+1}T_{i}T_{i+1}T_{i} \; \; \; \; (i=0, i = n-1) 
\end{align*}
and the quadratic relations
\begin{align*}
(T_{0} + 1)(T_{0} + cd/q) &= 0, \\
(T_{i} + 1)(T_{i}-t) &= 0, \; \; \; i \neq 0,n, \\
(T_{n} + 1)(T_{n} + ab) &= 0.
\end{align*}

\end{definition}

Recall that, by the Noumi representation (see \cite{Sahi, Stok1}), there is an action of $H$ on the vector space of Laurent polynomials $\mathbb{C}(q^{1/2},t,a,b,c,d)[x_{1}^{\pm 1}, \dots, x_{n}^{\pm 1}]$ (here $x_{1}, \dots, x_{n}$ are $n$ independent indeterminates) as follows:
\begin{align*}
T_{0}f &= -(cd/q)f + \frac{(1-c/x_{1})(1-d/x_{1})}{1-q/x_{1}^{2}}(f^{s_{0}}-f) \\
T_{i}f &= tf + \frac{x_{i+1}-tx_{i}}{x_{i+1}-x_{i}}(f^{s_{i}}-f) \; \; \; \; \; (\text{for } 0<i<n)\\
T_{n}f &= -abf + \frac{(1-ax_{n})(1-bx_{n})}{1-x_{n}^{2}}(f^{s_{n}}-f),
\end{align*}
where $f^{s_{0}}(x_{1}, \dots, x_{n}) = f(q/x_{1},x_{2}, \dots, x_{n})$, $f^{s_{i}}(x_{1}, \dots, x_{n}) = f(x_{1}, \dots, x_{i-1},x_{i+1},x_{i},x_{i+2}, \dots, x_{n})$ for $0<i<n$, and $f^{s_{n}}(x_{1}, \dots, x_{n}) = f(x_{1}, \dots, x_{n-1},1/x_{n})$.  Note that, for $0<i \leq n$, the action of $T_{i}$ on polynomials is independent of $q$; this will be crucial for the rest of the paper.  

We will denote the nonsymmetric Koornwinder polynomials in $n$ variables by $U_{\lambda}^{(n)}(x;q,t;a,b,c,d)$ \cite{NK, Sahi, Stok1, Stok2}.  Recall that these polynomials are indexed by compositions $\lambda$.  We will remind the reader how these polynomials are defined, but we must first set up some relevant notation.
Let $^{\iota}$ be the involution defined by
\begin{equation*}
^{\iota}: q \rightarrow q^{-1}, t \rightarrow t^{-1}, a \rightarrow a^{-1}, b \rightarrow b^{-1}, c \rightarrow c^{-1}, d \rightarrow d^{-1}, z^{\mu} \rightarrow z^{\mu}
\end{equation*}
and let $\bar{}$ be the involution defined by
\begin{equation*}
\bar{}: q \rightarrow q, t \rightarrow t, a \rightarrow a, b \rightarrow b, c \rightarrow c, d \rightarrow d, z^{\mu} \rightarrow z^{-\mu}.
\end{equation*}
Define the weight
\begin{multline} \label{nonsymmdensity}
\Delta_{K}^{(n)}(z;q,t;a,b,c,d) = \prod_{1 \leq i \leq n} \frac{(z_{i}^{2}, qz_{i}^{-2};q)}{(az_{i}, bz_{i}, cz_{i}, dz_{i}, aqz_{i}^{-1}, bqz_{i}^{-1}, cz_{i}^{-1},dz_{i}^{-1};q)} \\
\times \prod_{1 \leq i<j \leq n} \frac{(z_{i}z_{j}^{\pm 1},qz_{i}^{-1}z_{j}^{\pm 1};q)}{(tz_{i}z_{j}^{\pm 1}, qtz_{i}^{-1}z_{j}^{\pm 1};q)},
\end{multline}
i.e., the full nonsymmetric density, see \cite{RV}.  Note that $\Delta_{K}^{(n)}(z;q,1;1,-1,0,0) = 1$; this specialization is independent of $q$.  As in the symmetric case, when the parameters are clear from context, we will suppress them to make the notation easier.  Note the following formula for the nonsymmetric density at the specialization $q=0$:
\begin{equation} \label{nonsymmzero}
\prod_{1 \leq i \leq n} \frac{(1-z_{i}^{2})}{(1-az_{i})(1-bz_{i})(1-cz_{i})(1-dz_{i})(1-cz_{i}^{-1})(1-dz_{i}^{-1})} \prod_{1 \leq i<j \leq n} \frac{1-z_{i}z_{j}^{\pm 1}}{1-tz_{i}z_{j}^{\pm 1}}.
\end{equation}
We will write $\Delta_{K}^{(n)}(z;t;a,b,c,d)$ to indicate this particular limiting case.

With this terminology, consider the following inner product on functions of $n$ variables with parameters $q,t,a,b,c,d$:
\begin{equation*}
\langle f,g \rangle_{q} = \int_{T} f \bar{g}^{\iota} \Delta_{K}^{(n)}(z;q,t;a,b,c,d) dT.
\end{equation*}
Note that it is the constant term of $f \bar{g}^{\iota}\Delta_{K}^{(n)}$.  Also, denote by $\langle, \rangle_{0}$ the following inner product:
\begin{equation} \label{ipzero}
 \langle f,g \rangle_{0} = \int_{T} f \bar{g}^{\iota} \Delta_{K}^{(n)}(z;t;a,b,c,d) dT,
\end{equation}
involving the $q = 0$ degeneration of the full nonsymmetric Koornwinder weight as in (\ref{nonsymmzero}).  

Recall that the polynomials $\{U_{\mu}^{(n)}(x;q,t;a,b,c,d) \}_{\mu \in \mathbb{Z}^{n}}$ are uniquely defined by the following conditions:
\begin{align*}
&(i) \; \; U_{\mu} = x^{\mu} + \sum_{\nu \prec \mu} w_{\mu \nu} x^{\nu} \\
& (ii) \; \; \langle U_{\mu}, x^{\nu} \rangle_{q} = 0 \text{ if } \nu \prec \mu,
\end{align*}
where as usual we write $x^{\mu}$ for the monomial $x_{1}^{\mu_{1}}x_{2}^{\mu_{2}} \cdots x_{n}^{\mu_{n}}$.  

\begin{definition}\label{nonsymmpart}
For a partition $\lambda$ with $l(\lambda) \leq n$, define
\begin{equation*} 
E_{\lambda}^{(n)}(z;c,d) =  \prod_{\lambda_{i} > 0} z_{i}^{\lambda_{i}} (1-cz_{i}^{-1})(1-dz_{i}^{-1}).
\end{equation*}
\end{definition}

\subsection{Main Results}
We will first show that, under the assumption that $\lambda$ is a partition with $l(\lambda) \leq n$, $E_{\lambda}^{(n)}(z;c,d)$ is the $q=0$ limiting case of the nonsymmetric Koornwinder polynomial $U_{\lambda}^{(n)}(x;q,t;a,b,c,d)$.


\begin{theorem} \label{triang}
(Triangularity) The polynomials $E_{\lambda}^{(n)}(z;c,d)$ are triangular with respect to dominance ordering, i.e.,
\begin{equation*}
E_{\lambda}^{(n)}(z;c,d) = z^{\lambda} + \sum_{\mu \prec \lambda} c_{\mu}z^{\mu}
\end{equation*}
for all partitions $\lambda$.  
\end{theorem}
\begin{proof}
It is clear that $E_{\lambda}^{(n)}(z;c,d) = z^{\lambda} + (\text{dominated terms})$, since the term inside the product definition of $E_{\lambda}^{(n)}(z;c,d)$ is just $z_{i}^{\lambda_{i}} - (c + d)z_{i}^{\lambda_{i}-1} + cd z_{i}^{\lambda_{i}-2}$.
\end{proof}

\begin{theorem} \label{nonsymmeval}
We have the following constant term evaluation in the nonsymmetric case (with respect to $q=0$ limit of the nonsymmetric density as in (\ref{nonsymmzero}))
\begin{equation*}
\int_{T} \Delta_{K}^{(n)}(z;t;a,b,c,d) dT = \prod_{i=0}^{n-1} \frac{1}{(1-t^{i}ac)(1-t^{i}bc)(1-t^{i}cd)(1-t^{i}ad)(1-t^{i}bd)} \prod_{j=n-1}^{2n-2} (1-t^{j}abcd).
\end{equation*}
\end{theorem}

\begin{proof}
This follows from the proof of Theorem \ref{symmeval}, in particular recall (\ref{nonsymmeqn}).  
\end{proof}

\begin{theorem} \label{orthog} (Orthogonality)
Let $\lambda$ be a partition with $l(\lambda) \leq n$ and $\mu \in \mathbb{Z}^{n}$ a composition, such that $\mu \prec \lambda$.  Then we have $\langle E_{\lambda}^{(n)}(z;c,d), z^{\mu} \rangle_{0} = 0$.
\end{theorem}

\begin{proof} 
Fix $\lambda$ a partition.  First note that, by definition of the inner product $\langle \cdot, \cdot,  \rangle_{0}$ in (\ref{ipzero}) we have
\begin{multline*}
\langle E_{\lambda}^{(n)}(z;c,d), z^{\mu} \rangle_{0} \\= \int_{T} E_{\lambda}^{(n)}(z;c,d) z^{-\mu}\prod_{1 \leq i \leq n} \frac{(1-z_{i}^{2})}{(1-az_{i})(1-bz_{i})(1-cz_{i})(1-dz_{i})(1-cz_{i}^{-1})(1-dz_{i}^{-1})} \prod_{1 \leq i<j \leq n} \frac{1-z_{i}z_{j}^{\pm 1}}{1-tz_{i}z_{j}^{\pm 1}}.
\end{multline*}

We will first show $\langle E_{\lambda}^{(n)}(z;c,d), z^{\mu} \rangle_{0} = 0$ for all compositions $\mu$ satisfying the following two properties:

\noindent \textbf{Condition (i)} $\mu \stackrel{\text{lex}}{<} \lambda$, so in particular there exists $1 \leq i \leq n$ such that $\mu_{1} = \lambda_{1}, \dots, \mu_{i-1} = \lambda_{i-1}$ and $\mu_{i}< \lambda_{i}$.  

\noindent \textbf{Condition (ii)} $\lambda_{i} \neq 0$ (where $i$ is as in (i)).  

We mention that condition (ii) is necessary because of the difference between nonzero and zero parts of $\lambda$ in Definition \ref{nonsymmpart}; in particular if $\lambda_{i} = 0$ then one does not have the term $z_{i}^{\lambda_{i}}(1-cz_{i}^{-1})(1-dz_{i}^{-1})$ in $E_{\lambda}^{(n)}(z;c,d)$ (so that one still has the terms $1/(1-cz_{i}^{-1})(1-dz_{i}^{-1})$ in the product $E_{\lambda}^{(n)}(z;c,d)\Delta_{K}^{(n)}$) .
We give a proof by induction on $n$, the number of variables.  Note first that condition (ii) implies that $\lambda_{1}, \dots, \lambda_{i-1} \neq 0$.  Consider the case $n=1$.  Then in particular $i=1$ and conditions (i) and (ii) give $\mu_{1} < \lambda_{1} \neq 0$.  One can then compute
\begin{equation*}
\langle E_{\lambda}^{(n)}, z^{\mu} \rangle_{0} = \int_{T} z_{1}^{\lambda_{1} - \mu_{1}} \frac{(1-z_{1}^{2})}{(1-az_{1})(1-bz_{1})(1-cz_{1}) (1-dz_{1})} dT,
\end{equation*}
since $\lambda_{1} - \mu_{1} >0$ this is necessarily zero.  Now suppose the claim holds for $n-1$, we show it holds for $n$.  

We may restrict the $n$-dimensional integral $\langle E_{\lambda}(z), z^{\mu} \rangle_{0}$ to the contribution involving $z_{1}$, one computes it to be 
\begin{equation*}
\int_{T} z_{1}^{\lambda_{1} - \mu_{1}} \frac{(1-z_{1}^{2})}{(1-az_{1})(1-bz_{1})(1-cz_{1}) (1-dz_{1})} \prod_{j>1} \frac{1-z_{1}z_{j}^{\pm 1}}{1-tz_{1}z_{j}^{\pm 1}} dT.
\end{equation*}
If $i=1$, then $\lambda_{1} > \mu_{1}$ and this integral (and consequently the $n$-dimensional integral) are zero.  If $i>1$, then $\lambda_{1} = \mu_{1}$ and this integral is $1$.  In this case, one notes that the resulting $n-1$ dimensional integral is exactly:
\begin{equation*}
\int_{T} E_{\widehat{\lambda}}^{(n-1)}(z_{2}, \dots, z_{n}) z^{-\widehat{\mu}} \Delta_{K}^{(n-1)} dT,
\end{equation*}
where $\widehat{\lambda} = (\lambda_{2}, \dots, \lambda_{n})$ and $\widehat{\mu} = (\mu_{2}, \dots, \mu_{n})$.  Note that conditions (i) and (ii) hold for $\widehat{\mu}$ and $\widehat{\lambda}$, and since this is the $n-1$ variable case we may appeal to the induction hypothesis.  Thus, the above integral is zero; consequently $\langle E_{\lambda}^{(n)}, z^{\mu} \rangle = 0$ as desired.  

Finally, it remains to show that $\mu \prec \lambda$ implies conditions (i) and (ii).  Recall that there are two cases for $\mu \prec \lambda$.  In case 1), note that we have $\mu \leq \mu^{+} < \lambda$ with respect to the dominance ordering, so $\mu < \lambda$.  This implies $\mu \stackrel{\text{lex}}{<} \lambda$ by Lemma \ref{order}.  In case 2), it is clear.  Now we show condition (ii).  Suppose for contradiction that $\lambda_{i} = 0$, so that $\mu_{1} = \lambda_{1} \geq 0 , \dots, \mu_{i-1} = \lambda_{i-1} \geq 0$ and $\mu_{i} < \lambda_{i} = 0$ and $\lambda_{k} = 0$ for all $i<k \leq n$.  Then note that $\sum_{k=1}^{i} (\mu^{+})_{k} > \sum_{k=1}^{i} \lambda_{k}$, which contradicts $\mu^{+} \leq \lambda$.  Thus, we must have $\lambda_{i} \neq 0$ as desired.  
\end{proof}

\begin{theorem} \label{nonsymmnorm}
Let $\lambda$ be a partition with $l(\lambda) \leq n$, then 
\begin{multline*}
\langle E_{\lambda}^{(n)}(z;c,d), E_{\lambda}^{(n)}(z;c,d) \rangle_{0} = \langle E_{\lambda}^{(n)}(z;c,d), z^{\lambda} \rangle_{0} \\
=\prod_{i=0}^{m_{0}(\lambda)-1} \frac{1}{(1-t^{i}ac)(1-t^{i}bc)(1-t^{i}cd)(1-t^{i}ad)(1-t^{i}bd)} \prod_{j=m_{0}(\lambda)-1}^{2m_{0}(\lambda)-2} (1-t^{j}abcd)
\end{multline*}
\end{theorem}

\begin{proof}
The first equality follows from Theorems \ref{triang} and \ref{orthog}.  For the second equality, we use arguments similar to those used in the proof of Theorem \ref{orthog}.  We first note that 
\begin{multline*}
\langle E_{\lambda}^{(n)}(z;c,d), z^{\lambda} \rangle_{0} \\= \int_{T} E_{\lambda}^{(n)}(z;c,d) z^{-\lambda}\prod_{1 \leq i \leq n} \frac{(1-z_{i}^{2})}{(1-az_{i})(1-bz_{i})(1-cz_{i})(1-dz_{i})(1-cz_{i}^{-1})(1-dz_{i}^{-1})} \prod_{1 \leq i<j \leq n} \frac{1-z_{i}z_{j}^{\pm 1}}{1-tz_{i}z_{j}^{\pm 1}}.
\end{multline*}
One can integrate with respect to $z_{1}$, holding the remaining variables fixed; the reader can verify that the result is $\langle E_{\widehat \lambda}^{(n-1)}(z;c,d), z^{\widehat \lambda}\rangle_{0}$, where $\widehat \lambda = (\lambda_{2}, \dots, \lambda_{n})$.  Iterating this argument shows that
\begin{equation*}
\langle E_{\lambda}^{(n)}(z;c,d), z^{\lambda} \rangle_{0} = \int_{T} \Delta_{K}^{(m_{0}(\lambda))} dT.
\end{equation*}
By Theorem \ref{nonsymmeval}, this is equal to 
\begin{equation*}
\prod_{i=0}^{m_{0}(\lambda)-1} \frac{1}{(1-t^{i}ac)(1-t^{i}bc)(1-t^{i}cd)(1-t^{i}ad)(1-t^{i}bd)} \prod_{j=m_{0}(\lambda)-1}^{2m_{0}(\lambda)-2} (1-t^{j}abcd),
\end{equation*}
as desired.
\end{proof}

In the next theorem, we will prove that these polynomials $E_{\lambda}^{(n)}(z;c,d)$ are indeed the nonsymmetric Koornwinder polynomials indexed by a partition in the limit $q \rightarrow 0$.  In order to do this, we will show that the limit is actually well-defined, and that polynomials satisfying the above triangularity and orthogonality conditions are uniquely determined.

\begin{theorem} \label{qlim}
The $q \rightarrow 0$ limit of the nonsymmetric Koornwinder polynomials are well-defined.  Moreover, for a partition $\lambda$, $\displaystyle \lim_{q \rightarrow 0} U_{\lambda}^{(n)}(z; q,t; a,b,c,d) = E_{\lambda}^{(n)}(z;c,d)$.  
\end{theorem}

\begin{proof}
By virtue of the triangularity condition for nonsymmetric Koornwinder polynomials, we can write $U_{\lambda}^{(n)}(x;q,t;a,b,c,d) = x^{\lambda} + \sum_{\mu \prec \lambda} c_{\mu} x^{\mu}$.  The orthogonality condition for these polynomials then gives
\begin{equation*}
\langle U_{\lambda}^{(n)}(x;q,t;a,b,c,d), x^{\nu} \rangle_{q} = \Big\langle x^{\lambda} + \sum_{\mu \prec \lambda} c_{\mu} x^{\mu}, x^{\nu}\Big\rangle_{q} = 0
\end{equation*}
for all $\nu \prec \lambda$.  By linearity of the inner product, one can rewrite this as 
\begin{equation*}
\Big\langle \sum_{\mu \prec \lambda} c_{\mu} x^{\mu}, x^{\nu} \Big\rangle_{q} = -\langle x^{\lambda}, x^{\nu} \rangle_{q}
\end{equation*}
for all $\nu \prec \lambda$.  
So the coefficient vector $\vec{c}$ is uniquely determined by the equation
\begin{equation} \label{GSeqn}
\vec{c} A_{[\prec \lambda]} = - \vec A_{\lambda},
\end{equation}
where $A_{[\prec \lambda]}$ is the inner product matrix of the monomials $\{x^{\nu}\}_{\nu \prec \lambda}$, and $\vec A_{\lambda}$ is the column vector of inner products of $x^{\lambda}$ with $\{x^{\nu} \}_{\nu \prec \lambda}$.  We need to check that the limit as $q \rightarrow 0$ of the entries of $\vec{c}$ is well-defined.  It suffices to show that $\lim_{q \to 0} \det A_{[\prec \lambda]}$ is not equal to zero.  We will exhibit a specialization of the parameters $(t,t_{0}, \dots, t_{3})$ for which this limit is nonzero.  To this end, consider the following specialization of parameters: $(t = 1, t_{0} = 1, t_{1} = -1, t_{2} = t_{3} = 0)$; note that it is independent of $q$.  Under this specialization, the inner product weight becomes $\Delta_{K} = 1$, and the inner product for monomials is $\langle x^{\lambda}, x^{\mu} \rangle_{q} = \delta_{\lambda, \mu}$.  In particular, the matrix $A_{[<\lambda]}$ is the identity matrix and it has determinant equal to one.  We write $\lim_{q \rightarrow 0} \vec{c}$ for the $q \rightarrow 0$ limit of the entries of the vector $\vec{c}$.    

Now, write $E_{\lambda}^{(n)}(x;c,d) = x^{\lambda} + \sum_{\mu \prec \lambda}  d_{\mu}x^{\mu}$, using Theorem \ref{triang}.  Then, using Theorem \ref{orthog}, exactly as in the previous paragraph we find $\vec{d}$ satisfies $\vec{d} A_{[<\lambda]}^{0} = -A_{\lambda}^{0}$, where $\cdot^{0}$ denotes the specialization of the inner product for monomials at $q \rightarrow 0$.  But since the vectors $\vec{c}$ and $\vec{d}$ are uniquely determined via the method discussed above, and the matrices here are the $q \rightarrow 0$ specializations of the ones in (\ref{GSeqn}), we must have $\lim_{q \rightarrow 0} \vec{c} = \vec{d}$, as desired.  
\end{proof}

We now use Definition \ref{nonsymmpart} to extend to the case where $\lambda \in \Lambda$ via a particular recursion.  We will first need to define some relevant rational functions in $t,a,b,c,d$.  

\begin{definition}
Define
\begin{equation*}
n_{\lambda} = -|\{ l<i: \lambda_{l} = -1 \text{ or } 0 \}| - 2|\{l>i+1: \lambda_{l} = 0 \}| - 1
\end{equation*}
and
\begin{equation*}
r_{\lambda} = m_{-1}(\lambda) + m_{0}(\lambda) - 1;
\end{equation*}
\end{definition}
they are statistics of the composition $\lambda$.  We use this to define rational functions $\{p_{i}(\lambda)\}_{1 \leq i \leq n}$ and $\{q_{i}(\lambda)\}_{1 \leq i \leq n}$ as follows
\begin{equation*} p_{n}(\lambda) = 
\begin{cases}
-ab-1, & \text{if } \lambda_{n} < -1 \\
-ab - 1 + abcdt^{2r_{\lambda}}, & \text{if } \lambda_{n} = -1 \\
0, & \text{if } \lambda_{n} > 1 \\
-abcdt^{2r_{\lambda}}, & \text{if } \lambda_{n} =1 \\
-ab,& \text{if } \lambda_{n} = 0
\end{cases}
\end{equation*}
and for $1 \leq i \leq n-1$
\begin{equation*}
p_{i}(\lambda) = 
\begin{cases}
t-1, & \text{if } \lambda_{i} < \lambda_{i+1} \text{ and } (\lambda_{i}, \lambda_{i+1}) \neq (-1,0) \\
\frac{(1-t)t^{n_{\lambda}}}{abcd - t^{n_{\lambda}}}, & \text{if } (\lambda_{i}, \lambda_{i+1}) = (-1,0) \\
0, & \text{if } \lambda_{i} > \lambda_{i+1} \text{ and } (\lambda_{i}, \lambda_{i+1}) \neq (0,-1) \\
\frac{(t-1)abcd}{abcd - t^{n_{\lambda}}}, & \text{if } (\lambda_{i}, \lambda_{i+1}) = (0,-1) \\
t, & \text{if } \lambda_{i} = \lambda_{i+1}.
\end{cases}
\end{equation*}
Similarly, define
\begin{equation*}
q_{n}(\lambda) = 
\begin{cases}
 -ab, & \text{if } \lambda_{n} < 0 \\
 0, & \text{if } \lambda_{n} = 0 \\
 1, & \text{if } \lambda_{n} > 1 \\
 1+cdt^{2r_{\lambda}}(-ab-1+abcdt^{2r_{\lambda}}), & \text{if } \lambda_{n} = 1
 \end{cases}
\end{equation*}
and for $1 \leq i \leq n-1$
\begin{equation*}
q_{i}(\lambda) =
\begin{cases}
t, & \text{if }\lambda_{i}<\lambda_{i+1}\\
0,& \text{if } \lambda_{i} = \lambda_{i+1} \\
1,& \text{if } \lambda_{i} > \lambda_{i+1} \text{ and } (\lambda_{i}, \lambda_{i+1}) \neq (0,-1) \\
1-\frac{(1-t)^{2}abcdt^{n_{\lambda}-1}}{(abcd-t^{n_{\lambda}})^{2}},& \text{if } (\lambda_{i}, \lambda_{i+1}) = (0,-1).
\end{cases}
\end{equation*}

\begin{definition}
For $\lambda \in \Lambda$ with $l(\lambda) \leq n$, define $E_{s_{i}\lambda}^{(n)}(z;t;a,b,c,d)$ (for $1 \leq i \leq n$) by the following recursion
\begin{equation} \label{recurqlim}
T_{i}E_{\lambda} = p_{i}(\lambda)E_{\lambda} + q_{i}(\lambda)E_{s_{i}\lambda}
\end{equation}
where $p_{i}(\lambda), q_{i}(\lambda)$ are the rational functions of the previous definition, and for $\lambda$ a partition $E_{\lambda}^{(n)}$ is given by Definition \ref{nonsymmpart}.  
\end{definition}

\begin{theorem} \label{fullqlim}
For any $\lambda \in \Lambda$, $E_{\lambda}^{(n)}(z;t;a,b,c,d)$ is well-defined and we have
\begin{equation*}
\lim_{q \rightarrow 0} U_{\lambda}^{(n)}(z;q,t;a,b,c,d) = E_{\lambda}^{(n)}(z;t;a,b,c,d).
\end{equation*}
\end{theorem}

\begin{proof} 
The case when $\lambda$ is a partition has been established by Theorem \ref{qlim}.  The rest of the result is obtained by showing that \cite{Stok2} Proposition 6.1 admits the limit $q \rightarrow 0$ and that recursion in fact becomes (\ref{recurqlim}) in this limit.  As mentioned above, it is crucial that the operators $T_{i}$ for $1 \leq i \leq n-1$ are independent of $q$; these are the operators that appear here.  We note that the parameters must be translated according to the following reparametrization:
\begin{equation*}
\{a,b,c,d,t \} \leftrightarrow \{ t_{n}\check{t_{n}}, -t_{n}\check{t_{n}}^{-1}, t_{0}\check{t_{n}}q^{1/2}, -t_{0}\check{t_{0}}^{-1}q^{1/2}, t^{2} \};
\end{equation*}
in particular, this reparametrization yields $T_{0} = t_{0}Y_{0}, T_{n} = t_{n}Y_{n}, T_{i} = tY_{i}$ (for $1 \leq i \leq n-1$), where $\{Y_{i}\}_{0 \leq i \leq n}$ are the Hecke operators of \cite{Stok2}.  

It is a computation to directly verify that the limits exist in the cases $\lambda_{n} \leq 0$ and $\lambda_{i} \leq \lambda_{i+1} (1 \leq i \leq n-1)$.  We then apply $T_{n}$ and $T_{i}(1 \leq i \leq n-1)$ to the resulting recursions and use the quadratic relations
\begin{equation*}
T_{n}^{2} = -ab - T_{n} -abT_{n}
\end{equation*}
and 
\begin{equation*}
T_{i}^{2} = t- T_{i} +tT_{i}
\end{equation*}
and simplify to obtain the recursion in the remaining cases $\lambda_{n} > 0$ and $\lambda_{i} > \lambda_{i+1}$.
\end{proof}

As a byproduct of orthogonality for the $\{U_{\lambda}^{(n)}(z;q,t;a,b,c,d)\}_{\lambda \in \Lambda}$ we obtain the complete orthogonality for the $q=0$ limiting case.

\begin{corollary}
Let $\lambda, \mu \in \mathbb{Z}^{n}$ be compositions.  If $\lambda \neq \mu$, we have $\langle E_{\lambda}^{(n)}, E_{\mu}^{(n)}\rangle_{0} = 0$.  If $\mu \prec \lambda$, then we have $\langle E_{\lambda}^{(n)}, z^{\mu} \rangle_{0} = 0$.
\end{corollary}
\begin{proof}
This follows from the orthogonality for the $q$-nonsymmetric Koornwinder polynomials and Theorem \ref{fullqlim}.  
\end{proof}

\end{document}